\documentclass[a4paper]{scrartcl}

\usepackage{graphicx}
\usepackage[utf8]{inputenc}
\usepackage[english]{babel}

\usepackage{amsmath,amssymb,enumerate,url,appendix,tabularx,stmaryrd,mathrsfs}

\usepackage{todonotes,comment,afterpage}
\usepackage{subcaption}
\usepackage{framed}
\usepackage{mathtools}
\usepackage{fullpage}

\usepackage{amsthm}
\usepackage{graphicx}

\theoremstyle{plain}
\newtheorem{theorem}{Theorem}[section]
\newtheorem{lemma}[theorem]{Lemma}
\newtheorem{corollary}[theorem]{Corollary}
\newtheorem{proposition}[theorem]{Proposition}

\newtheorem{assumption}[theorem]{Assumption}

\newtheorem{definition}[theorem]{Definition}

\newtheorem{problem}[theorem]{Problem}

\newtheorem{remark}[theorem]{Remark}

\numberwithin{equation}{section}
\numberwithin{figure}{section}
\numberwithin{table}{section}

\newcommand{\R}{\mathbb{R}}
\newcommand{\C}{\mathbb{C}}
\newcommand{\N}{\mathbb{N}}
\newcommand{\Z}{\mathbb{Z}}
\newcommand{\bigO}{\mathcal{O}}

\newcommand{\abs}[1]{\left|#1\right|}
\newcommand{\norm}[1]{\left\|#1\right\|}

\newcommand{\support}{\operatorname{supp}}
\newcommand{\dom}{\operatorname{dom}}
\newcommand{\laplace}{\Delta}

\def\TT{\mathcal{T}}
\def\SS{\mathcal{S}}

\def\QQ{\mathcal{Q}}
\def\ii{i}

\def\fdiv{\operatorname{div}}
\def\id{\mathrm{I}}

\def\LL{\mathcal{L}}

\def\UU{\mathscr{U}}
\def\VV{\mathscr{V}}
\def\WW{\mathscr{W}}

\def\sg{\mathcal{E}}

\def\UUdy{\UU^h_{y}}

\def\UUdxyt{\UU^h_{x,y,t}}

\def\trace{\operatorname{tr}}
\def\CC{\mathcal{C}}
\def\YY{\mathcal{Y}}

\def\solve{\mathcal{G}}

\def\AA{\mathcal{A}}

\newcommand{\HH}{\mathring{H}^1(y^\alpha,\CC)}
\def\HHy{\mathbb{V}_h^{\YY}}
\def\HHx{\mathbb{V}_h^{\mathcal{X}}}
\def\HHxy{\mathbb{V}_h^{\mathcal{X},\YY}}
\def\HHxyt{\mathbb{V}_h^{\mathcal{X},\YY,\mathcal{T}}}

\def\HHnd{H^1(y^{\alpha},D)}

\newcommand{\ltwoprodX}[2]{\left(#1,#2\right)_{L^2(\Omega)}}

\def\lifting{\mathscr{D}}

\renewcommand{\Re}{\operatorname{Re}}
\renewcommand{\Im}{\operatorname{Im}}

\newcommand{\eremk}{\hbox{}\hfill\rule{0.8ex}{0.8ex}}

\def\zhf{z_{\text{hf}}}

\iffalse 
\usepackage{tikz}
\usetikzlibrary{shapes}
\usepackage{pgfplots}
\usetikzlibrary{calc}
\newcommand{\includeTikzOrEps}[1]{\tikzexternalenable \tikzsetnextfilename{#1}  {\include{figures/#1}} \tikzexternaldisable}

\usetikzlibrary{external}
\tikzexternaldisable

\else
\newcommand{\includeTikzOrEps}[1]{\includegraphics{figures_pdf/#1}}
\fi

\title{ $hp$-FEM for the fractional heat equation}

\author{Jens Markus Melenk\thanks{Institute for Analysis und Scientific Computing, Technische Universit\"at Wien,
    Wiedner Hauptstrasse 8-10, 1040 Vienna, Austria. melenk@tuwien.ac.at} \and
  Alexander Rieder\thanks{ Institute for Analysis und Scientific Computing, Technische Universit\"at Wien,
   Wiedner Hauptstrasse 8-10, 1040 Vienna, Austria. alexander.rieder@tuwien.ac.at}  
}

\date{\today}

\begin{document}
\maketitle
\begin{abstract}
  \vspace{-0.5cm}
We consider a time dependent problem generated by a nonlocal operator in space.
Applying a discretization scheme based on $hp$-Finite Elements and a Caffarelli-Silvestre extension
we obtain a semidiscrete semigroup. The discretization in time is carried out by using $hp$-Discontinuous Galerkin
based timestepping.
We prove exponential convergence for such a method in an abstract framework for the discretization in the
original domain $\Omega$.
\end{abstract}

\section{Introduction}
For stationary fractional diffusion, numerical techniques have recently been proposed that 
  provide exponential convergence of the error with respect to the computational effort, \cite{tensor_fem, tensor_fem_on_polygons}.
  The construction is based on $hp$-Finite Elements on appropriate geometric meshes.
  The purpose of the present article is to generalize these techniques 
  to the time dependent setting. 
  We consider the discretization of the time dependent problem~\eqref{eq:model_problem}, generated by a fractional power of an elliptic operator.
  The spatial discretization of the nonlocal operator is based on a reformulation
  using the Caffarelli-Silvestre extension, for which an $hp$-Finite Element discretization
  (FEM) is employed. The discretization in time is then carried out by a Discontinuous Galerkin method in the
  spirit of~\cite{schoetzau_schwab}  of either fixed order or in its $hp$ version.
  Our analysis hinges on two conditions,
  one related to stable liftings of the initial condition and the second one related to the ability to approximate solutions 
  of singularly perturbed problems.
  
  After establishing the abstract framework, we work out the case of $hp$-FEM in
    the special case of 1D or 2D with analytic data and geometry and show that the basic Assumptions~\ref{ass:approx_of_unitial_condtiion} and~\ref{ass:approx_of_HHx}
    are satisfied on appropriate geometric meshes. The reduction of scope to smooth geometries and at most 2D mainly is done to keep the presentation 
  to a reasonable length; we expect that it is possible to establish the assumptions of the abstract framework also for the case  of polygons or
  $\Omega \subseteq {\mathbb R}^d$, $d > 2$. 

Discretization schemes for the same model problem have already appeared in the literature. 
In~\cite{bonito_pasciak_parabolic}, the approximation is done by applying numerical quadrature to the Dunford-Taylor representation of
the solution and using a low-order finite element method in space. 
The idea of treating the extension problem via finite elements is already well established for the case of 
elliptic problems, e.g.~\cite{pde_approach_apriori} for the low-order FEM or \cite{mpsv17} as well as \cite{tensor_fem}
for using $hp$-based discretizations.
The use of an extension problem in order to discretize a time-dependent problem was used in~\cite{pde_frac_parabolic},
focusing on low order finite elements and time-stepping schemes, but allowing also for fractional time derivatives.
In the context of wave equations, such a discretization was recently analyzed in~\cite{bo_waves}.

When dealing with parabolic problems, it is well-known that, if the initial condition does not satisfy certain compatibility conditions,
so called startup singularities form. They need to be accounted for  in the numerical method. We rigorously prove 
that, as long as the meshes are designed in a proper way, our discretization scheme delivers exponential convergence rate
for the spatial discretization and optimal convergence rate in time, i.e., optimal order for fixed order timestepping like implicit Euler and 
exponential convergence for the $hp$-DG based method.

The paper is structured as follows: Section~\ref{sect:model_problem} presents the model problem and the functional analytic setting.
In Section~\ref{sect:space_disc}, we then perform a first discretization step with respect to the spatial variables. This yields
a continuous in time/discrete in space approximation. In order to prove exponential convergence for this discretization,
we take a small detour in Section~\ref{section:elliptic_problem} to analyze an auxiliary elliptic problem. This problem  will allow us to
lift a representation formula from the domain $\Omega$ to the extended cylinder $\Omega \times \R_+$ while allowing to
reuse the techniques developed in~\cite{tensor_fem}. These preparations then allow us to prove exponential convergence for the
space discretization in Section~\ref{sect:semidiscretization_2}.
The discretization in time is then carried out in Section~\ref{sect:time_discretization} yielding a fully discrete scheme.
This scheme was implemented and Section~\ref{sect:numerics} confirms the exponential convergence. 
The appendices provide results that could not readily be cited from the literature:  Appendix~\ref{appendix:oned_singular_perturbations} generalizes
results on $hp$-FEM for singularly perturbed problems to the case of complex perturbation parameters. Appendix~\ref{sect:liftings_and_interpolation}
is concerned with the lifting of piecewise polynomials in $\Omega$ to piecewise polynomials on the cylinder $\Omega\times \R_+$ in a stable way.

We also would like to point out that using the Caffarelli-Silvestre extension is not the only approach to discretize
the nonlocal operator which is able to yield an exponentially convergent scheme. We mention
schemes based on sinc-quadrature and the Balakrishnan or Riesz-Dunford formulations
of the fractional Laplacian (see~\cite{bonito_pasciak_parabolic}).  We expect that it is possible to
combine such a scheme with $hp$-FEM in the space discretization and by combining \cite{bonito_pasciak_parabolic} with
the techniques laid out in this paper it should be possible to show exponential convergence.

We close with a remark on notation. We write $A \lesssim B$ to mean there exists a constant $C>0$, which is independent of the
main quantities of interest, i.e., mesh size or polynomial degree used, etc., such that $A \leq C B$. We write $A \sim B$
to mean $A \lesssim B$ and $B \lesssim A$. The exact dependencies of the implied constant is specified in the context.

\section{Model problem}
\label{sect:model_problem}
Let $\Omega \subset \R^d$ be a bounded Lipschitz domain.
We consider the following model problem for $s \in (0,1)$:
\begin{subequations}
\label{eq:model_problem}
\begin{align}
  \dot{u}(t)+\LL^{s} u(t)&= f(t)  \quad &&\text{in $\Omega$, $\forall t > 0$}\\
  u(\cdot,t) &= 0  \quad &&\text{on $\Gamma$,  $\forall t > 0$}
\end{align}
\end{subequations}
with initial condition $u(0)=u_0$ and right-hand side $f: \Omega \times \R_+ \to \R$.
We assume that the initial condition and right-hand side are analytic but do not require any compatibility or boundary conditions.

The operator $\LL u:=-\fdiv(A \nabla u) + c u$ is a linear, elliptic and self-adjoint differential operator,
where we assume that  $A \in L^{\infty}\big(\Omega,\R^{d\times d}\big)$ is uniformly SPD in $\Omega$ and $c \in L^{\infty}(\Omega)$ satisfies $c\geq 0$.
The fractional power $\LL^s$ is defined using the spectral decomposition
\begin{align}
   \LL^s u:=\sum_{j=0}^{\infty}{\mu_j^{s} (u,\varphi_j)_{L^2(\Omega)} \varphi_j}, \label{eq:def_fractional_laplacian}
\end{align}
where $(\mu_j,\varphi_j)_{j\in\N_0}$ are eigenvalues and eigenfunctions of the operator $\LL$ with homogeneous Dirichlet boundary conditions.

Using the Caffarelli-Silvestre extension one can localize the nonlocal operator $\LL^s$ and rewrite \eqref{eq:model_problem} in the following form
with $\alpha:=1-2s$:
\begin{subequations}
\begin{align}
  -\fdiv\left(y^\alpha A \nabla \UU \right)+  y^\alpha c \UU &= 0 \qquad && \text{on $\CC \times (0,T)$},  \label{eq:extended_problem_pde}\\
  d_{s} \trace{\dot{\UU}}+\partial_{\nu}^\alpha \UU &= d_s f \qquad &&\text{on $\Omega \times \{0\} \times (0,T)$},  \label{eq:extended_problem_bc}\\
  \UU &= 0 \qquad &&\text{ on $\partial_L \CC \times (0,T)$}.
\end{align}
\end{subequations}

Here $\CC$ denotes the cylinder $\Omega \times \R_+$, $d_s:=2^{\alpha} \Gamma(1-s)/\Gamma(s)$.
The lateral boundary is defined as $\partial_L \CC:=\partial \Omega \times \R_+$ and 
$$
\partial_{\nu}^{\alpha} \UU := - \lim_{y\to 0^+}  y^{\alpha} \partial_{y} \UU(\cdot,y), \qquad \text{and } \qquad \trace \UU:=\UU(\cdot,0)
$$
is the conormal derivative and boundary trace at $y=0$ respectively. The connection to $u$ is then given by $\trace{\UU}(t)=u(t)$.

In order to treat this extended problem, we introduce the following weighted Sobolev spaces:
\begin{align}
  L^2(y^\alpha,D)&:=\Big\{w: w \text{ is measurable and }  {\int_{D}}y^\alpha \abs{w}^2\,<\infty \Big\}, \\
  \HHnd&:=\left\{w \in L^2(y^{\alpha},D): \abs{\nabla w} \in L^2(y^\alpha,D) \right\}, \\
  \mathring{H}^1(y^{\alpha},D)&:=\left\{ w \in H^1(y^\alpha,D): u=0 \text{ on } \partial_L \CC \right\}.
\end{align}
The space $\HH$ is equipped with the norm $\displaystyle\norm{\UU}_{\HH}^2:=\int_{\mathcal{C}}{y^\alpha \abs{\nabla \UU}^2 }$.

We also define the bilinear form corresponding to the weak form of~\eqref{eq:extended_problem_pde} as:
\begin{align*}
  \AA(\UU,\VV):=\int_{\CC}{y^\alpha \big(A  \nabla \UU\big)\cdot \nabla \VV  + y^\alpha c \, \UU \VV}.
\end{align*}

Throughout this paper, we will make use of fractional Sobolev  and interpolation spaces. We 
define for two Banach spaces $X_1 \subseteq X_0$ with continuous embedding 
and $\theta \in (0,1)$: 
\begin{align*}
   \norm{u}^2_{[X_0,X_1]_{\theta,2}}
  &:= \int_{t=0}^\infty{ t^{-2\theta} \Big( \inf_{v \in X_1} \|u - v\|_{0} + t \|v\|_1\Big)^2 \frac{dt}{t}}, \\
    \left[X_0,X_1\right]_{\theta,2}&:=\big\{u \in X_0: \norm{u}_{[X_0,X_1]_{\theta,2}} < \infty \big\}.
\end{align*} 
For the endpoints $\theta \in \{0,1\}$ we set $[X_0,X_1]_{0,2}:=X_0$ and $[X_0,X_1]_{1,2}:=X_1$.
Fractional Sobolev spaces with and without zero boundary conditions are defined as
\begin{align*}
  \widetilde{H}^{s}(\Omega):=\left[L^2(\Omega), H_0^1(\Omega)\right]_{s,2}, \quad \text{and}\quad {H}^{s}(\Omega)&:=\left[L^2(\Omega), H^1(\Omega)\right]_{s,2}.
\end{align*}
The boundary condition in~\eqref{eq:model_problem} is understood in the sense of $u(t) \in \widetilde{H}^s(\Omega)$ for all $t >0$. That is, 
for $s<1/2$ no boundary condition is imposed, while for $s>1/2$ it is imposed in the sense of traces. For $s=1/2$ the boundary 
condition is imposed as membership in the Lions-Magenes space, often also denoted $H^{1/2}_{00}(\Omega)$.

Sometimes it is useful to work with a different scale of spaces,  characterized using the eigendecomposition of $\LL$,
as
\begin{align*}
  \mathbb{H}^s(\Omega):=\left\{u \in L^2(\Omega): \sum_{j=0}^{\infty}{\mu_j^s |{\ltwoprodX{u}{\varphi_j}}|^2} < \infty \right\}.
\end{align*}
For $s \in [0,1]$, the spaces coincide, i.e.,
$ \widetilde{H}^{s}(\Omega) = \mathbb{H}^s(\Omega)$ with equivalent norms.  

We consider the discretization in two separate steps. We semidiscretize in space and subsequently discretize in time, i.e.,
\begin{enumerate}
\item discretize in space using tensor product $hp$-FEM in $\Omega$ and the artificial variable $y$,
\item discretize in time by a discontinuous Galerkin method.
\end{enumerate}


\section{Discretization in space -- the semidiscrete scheme}
\label{sect:space_disc}
In this section we investigate the convergence of a semidiscrete semigroup to the solution of~\eqref{eq:model_problem}.
We consider finite dimensional subspaces $\HHx \subseteq H_0^1(\Omega)$ and $\{0 \} \ne \HHy\subseteq H^1(y^{\alpha},\R_+)$,
and set $\HHxy:=\HHx \otimes \HHy \subseteq \HH$ as our approximation space.
We keep most of our analysis as general as possible, but will provide concrete examples on how to
  implement these spaces in Sections~\ref{sect:disc_of_y} and~\ref{sect:disc_of_x}.
  Throughout the paper, we will write 
  $$\mathcal{N}_{\Omega}:=\operatorname{dim}(\HHx) \qquad \text{and} \qquad \mathcal{N}_{\YY}:=\operatorname{dim}(\HHy).$$
While we will give a detailed construction of $\HHy$ later on,
  for now we just assume that there exists $v\in \HHy$ with $v(0)=1$ in order to be able to solve Dirichlet problems.

We define the Galerkin approximation $\LL^s_h: \HHx \to \HHx$ to the operator $\LL^s$ via the relation:
\begin{align}
  \ltwoprodX{\LL_h^s u}{v}&:=\frac{1}{d_s}\AA(\lifting_h u,\lifting_h v),
\end{align}
where $\lifting_h: \HHx \to \HHxy$ denotes the solution to the following ``lifting problem'':
\begin{subequations}
  \label{eq:discrete_lifting}
\begin{align}
  \AA(\lifting_h u, \VV_h)&=0 \qquad \forall \VV_h \in \HHxy \text{ s.t. } \trace{\VV_h}=0, \label{eq:discrete_lifting_pde}\\
  \trace{\lifting_h u}&=u.
\end{align}
\end{subequations}
We also introduce the notation $\lifting$ for the solution to
\begin{subequations}
\label{eq:continuous_lifting}
\begin{align}
  \AA(\lifting u, \VV)&=0 \qquad \forall \VV \in \HH \text{ s.t. } \trace{\VV}=0,   \label{eq:continuous_lifting_pde}\\
  \trace{\lifting u}&=u.
\end{align}
\end{subequations}

\begin{remark}
  \label{remark:continuous_lifting_is_bounded}  
We note that by \cite[Proposition 2.5]{pde_approach_apriori} and the ellipticity of $\AA$, the operator $\lifting$ 
is bounded with respect to the $\widetilde{H}^s(\Omega) \to \HH$-norm.
For $hp$-FEM spaces, it is non-trivial to show that $\lifting_h$ is bounded, especially on anisotropic meshes.
See Appendix~\ref{sect:liftings_and_interpolation} for a related result in a simplified setting.
\eremk
\end{remark}

\begin{theorem}
  \label{thm:llh_generates_sg}
  The operator $-\LL_h^s$ is the generator of an analytic semigroup on $\left(\HHx, \norm{\cdot}_{L^2(\Omega)}\right)$.  
\end{theorem}
\begin{proof}
  The operator $\LL_h^s$ is symmetric due to the symmetry of $\AA$.
  By \cite[Section 2.5, Theorem 5.2]{pazy}, it remains to show the estimate
  \begin{align*}
    \norm{\left(\lambda \id  + \LL_h^s\right)^{-1} f}_{L^2(\Omega)} \leq \frac{M}{1+\abs{\lambda}} \norm{f}_{L^2(\Omega)}
  \end{align*}
  for $\Re(\lambda)\geq 0$ and a constant $M$ that is  independent of $u$ and $\lambda$. 
  It is easy to  see that $\left(\lambda \id  + \LL_h^s\right)^{-1} f=\trace{\UU_\lambda}$ where
  $\UU_{\lambda} \in \HHxy$ solves
  \begin{align*}
    \ltwoprodX{\lambda d_s \trace{\UU_\lambda}}{\trace{\VV_h}} + \AA({\UU_\lambda},{\VV_h})&=\ltwoprodX{d_s f}{\trace{\VV_h}} \qquad \forall \VV_h\in \HHxy.
  \end{align*}
  Existence of the inverse follows from the coercivity of the bilinear form on the left-hand side. The \textsl{a priori} estimate follows by testing 
  with $\VV_h:=\UU_\lambda$ to get:
  \begin{align*}
\Re(\lambda)  d_s \norm{\trace{\UU_\lambda}}^2 + \AA({\UU_\lambda},{\UU_\lambda})   \leq 
\left| \lambda  d_s \norm{\trace{\UU_\lambda}}^2 + \AA({\UU_\lambda},{\UU_\lambda})\right| &\leq d_s \norm{f}_{L^2(\Omega)}\norm{\trace{\UU_\lambda}}_{L^2(\Omega)}.
  \end{align*}
We use the continuity of the trace operator 
  $\norm{u}_{L^2(\Omega)}\lesssim \norm{\UU_\lambda}_{\HH}$ and distinguish between the 
  cases $|\lambda|$ large and $|\lambda|$ small to get the desired estimate.
\end{proof}

\begin{lemma}
  \label{lemma_llh_is_elliptic}
  If we equip the space $\HHx$ with the  norm
  \begin{align}
    \label{eq:def_hhx_norm}
    \norm{u}_{\HHx}:=\norm{\lifting_h u}_{\HH},
  \end{align}
  the operator $\LL_h^s$ is elliptic, i.e.,
  \begin{align*}
    c_1 \norm{u}^2_{\HHx}\leq  \ltwoprodX{\LL^s_h u}{u} \leq c_2 \norm{u}_{\HHx}^2.
  \end{align*}
  We also have the following estimate of the $\widetilde{H}^{s}(\Omega)$-norm:
  \begin{align*}
    c_3 \norm{u}^2_{\widetilde{H}^s(\Omega)}\leq \ltwoprodX{\LL^s_h u}{u}.
  \end{align*}
  The constants $c_i$ are independent of the spaces $\HHx$ and $\HHy$ and depend only on $\Omega$, $\alpha$, and $\LL$.
\end{lemma}
\begin{proof}
  By the trace estimate \cite[Proposition 2.5]{pde_approach_apriori}, we get
  \begin{align*}
    \norm{u}^2_{\widetilde{H}^s(\Omega)} &\lesssim \norm{\lifting_h u}^2_{\HH} \lesssim \AA(\lifting_h u,\lifting_h u) =d_s \ltwoprodX{\LL^s_h u}{u}.
  \end{align*}
  On the other hand we get:
  \begin{align*}
    \ltwoprodX{\LL^s_h u}{u} &= d_s^{-1} \AA(\lifting_h u,\lifting_h u) \lesssim \norm{\lifting_h u}_{\HH}^2 =  \norm{u}_{\HHx}^2. \qedhere
  \end{align*}
\end{proof}

The operator $\LL_h^s$ gives rise to the semidiscrete problem posed in $\HHx$:
\begin{subequations}
\label{eq:sd_semigroup}
\begin{align}  
  \dot{u}_h + \LL_h^s u_h &= \Pi_{L^2} f, \\
  u_h(0)&= u_{h,0},
\end{align}
\end{subequations}
where $\Pi_{L^2}: L^2(\Omega) \to \HHx$ denotes the $L^2$-orthogonal projection and $u_{h,0} \in \HHx$ denotes some approximation to the
initial condition.

By Duhamel's principle, $u$ and $u_h$ can be written as
\begin{align*}
  u(t)&=\sg(t) u_0 + \int_{0}^{t}{\sg(\tau)f(t-\tau) \,d\tau} \quad \text{ and }\quad 
  u_h(t)=\sg_h(t) u_0 + \int_{0}^{t}{\sg_h(\tau)f(t-\tau) \,d\tau},
\end{align*}
where $\sg: \R_+ \to \mathscr{B}(L^2(\Omega), L^2(\Omega))$ and
$\sg_h: \R_+ \to \mathscr{B}(\HHx, \HHx)$ are the semigroups generated by $-\LL^s$ and $-\LL^s_h$ respectively.

When considering the discrete flow for initial conditions without compatibility conditions,
the right spaces will be the following:
\begin{definition}
  Let $\beta \in (0,1)$.
  Recall that  the space $\HHx$ is equipped with the norm $\norm{u}_{\HHx}:=\norm{\lifting_h u}_{\HH}$.
  We define the interpolation spaces
  \begin{align*}
    \mathbb{V}^{\mathcal{X}}_{h,\beta}:=
    \left[ \left({\mathbb{V}}_{h}^{\mathcal{X}}, \norm{\cdot}_{L^2(\Omega)}\right), 
    \left({\mathbb{V}}_{h}^{\mathcal{X}},\norm{\cdot}_{\HHx}\right)\right]_{\beta,2}.
  \end{align*}
  We employ the convention $\norm{\cdot}_{\mathbb{V}^{\mathcal{X}}_{h,0}}=\norm{\cdot}_{L^2(\Omega)}$
  and $\norm{\cdot}_{\mathbb{V}^{\mathcal{X}}_{h,1}}=\norm{\cdot}_{\HHx}$ for the endpoints.
\end{definition}

Throughout this paper, we will work with abstract spaces $\HHx$. Exponential convergence of the numerical
method relies on the following 
Assumptions~\ref{ass:approx_of_unitial_condtiion} and \ref{ass:approx_of_HHx}: 
\begin{assumption}
  \label{ass:approx_of_unitial_condtiion}
  There exist constants $\beta$, $b$, $\mu>0$,
  such that for all $u_0$ that are analytic on a fixed neighborhood $\widetilde{\Omega}$ of $\overline{\Omega}$, 
  there exists a function $u_{h,0} \in \HHx$ and constants $C_{stab}, C_{approx}>0$ such that
  \begin{align*}
    \norm{u_{h,0}}_{\mathbb{V}^{\mathcal{X}}_{h,\beta}} \leq C_{stab}< \infty
    \qquad \text{ and } \qquad \norm{u_{0} - u_{h,0}}_{L^2(\Omega)} \leq  C_{approx} \,e^{-b \mathcal{N}_{\Omega}^\mu },
  \end{align*}
  where $\mathcal{N}_{\Omega}:=\operatorname{dim}\big(\HHx\big)$.  
\end{assumption}

When considering the Riesz-Dunford representation of $u$, the contour lies in the set of values for which $\LL - z$ is
elliptic. Therefore we consider the set of complex numbers up to a cone containing the part of the positive real axis 
for which $\LL - z$ is no longer elliptic.
\begin{definition}
\label{def:domain_of_ellipticity}
With the Poincar\'e constant $C_P$ of $\Omega$ and 
  fixed $0 < \varepsilon_0 < z_0 \leq \min\left(\frac{1}{2 C_{\text{P}}},1\right)^2$, we define 
  \begin{align*}
    \mathscr{S}:=\C \setminus \left[\left\{ z_0 + z: \abs{\operatorname{Arg}(z)}\leq \frac{\pi}{8}, \Re(z)\geq 0 \right\} 
    \cup B_{\varepsilon_0}(0)\right].
  \end{align*}
\end{definition}
\begin{remark}
  The set $\mathscr{S}$ is chosen in such a way that it contains the contour $\mathcal{C}$ used in the proof of Theorem~\ref{thm:best_approximation_HHx}.
  Namely, it contains the rays $\{r e^{\ii \pi/4}\,|\, r > 0\}$, $\{r e^{-\ii \pi/4}\,|\, r > 0\}$ as well as the circular arc $\{r_0 e^{\ii \theta}\,|\, \theta \in (-\pi/4,\pi/4)\}$,
  with $\varepsilon_0<r_0<z_0$ connecting the two rays.
  The ball $B_{\varepsilon_{0}}(0)$ is removed in order to avoid problems at $0$ when dividing by $z$.  See Figure~\ref{fig:contour}.
\eremk
\end{remark}

\begin{figure}
  \center
  \includeTikzOrEps{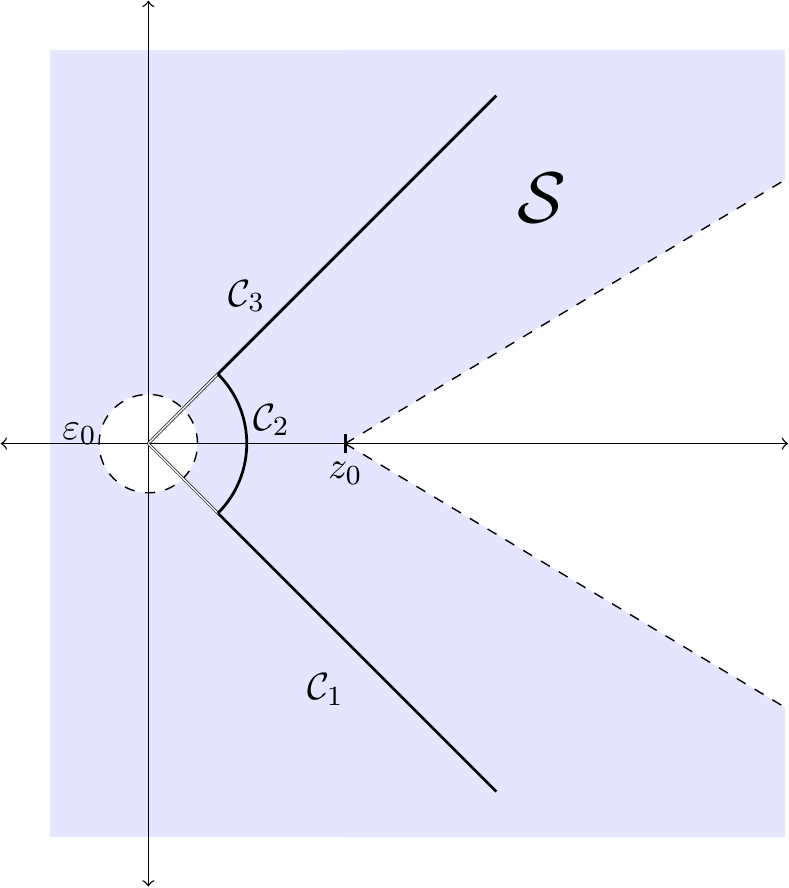}
  \caption{Geometric configuration of Definition~\ref{def:domain_of_ellipticity}}
  \label{fig:contour}
\end{figure}

  \begin{definition}
    \label{def:uniformly_analytic}
    A function $f: [0,T] \to L^\infty(\Omega)$ is said to be \emph{uniformly analytic} if:
    \begin{enumerate}[(i)]
    \item For all $t\in [0,T]$, $f(t)$ is analytic in a fixed neighborhood  $\widetilde{\Omega}$ of $\overline{\Omega}$,
    \item there exist constants $C_f,\gamma_f >0$, the \emph{analyticity constants of $f$}, such that for all $t\in [0,T]$ and $p \in \N_0$,
      \begin{align*}
        \norm{\nabla^{p} f(t)}_{L^{\infty}(\widetilde{\Omega})}&\leq C_f \gamma_f^p p !\,.
      \end{align*}     
    \end{enumerate}
  \end{definition}

The second assumption we have to make is that for a certain class of singularly perturbed
elliptic problems, the solution can be approximated exponentially well. We formalize this as follows.
\begin{assumption}
  \label{ass:approx_of_HHx}  
  A function space $\HHx$ is said to resolve down to the scale $\varepsilon >0 $ if
  for all $z \in \mathscr{S}$ with $\abs{z}^{-1/2} \geq \varepsilon$  
  and for all functions $f$ that are analytic on a fixed neighborhood $\widetilde{\Omega}$ of $\overline{\Omega}$, 
  the solutions to the elliptic problem
  \begin{align*}
    -z^{-1} \LL u  + u &= f
  \end{align*}
  can be approximated exponentially well from it. That is, 
  there exist constants $C(f)$, $b$ and $\mu>0$ such that
  \begin{align*}
    \inf_{v_h \in \HHx}\left[
    \abs{z}^{-1}\norm{\nabla u -  \nabla v}^2_{L^2(\Omega)} + \norm{u-v}^2_{L^2(\Omega)}\right]&\lesssim C(f) e^{-b \mathcal{N}_{\Omega}^{\mu}},
  \end{align*}
  where $\mathcal{N}_{\Omega}:=\operatorname{dim}(\HHx)$.
  The constant $C(f)$ may depend only on $\widetilde{\Omega}$, the analyticity constants of  $f$,
  on $A$, $c$, $\Omega$, $z_0$ and $\varepsilon_0$, while the constants $b$ and $\mu$ depend only on $A$, $c$, $\widetilde{\Omega}$, $\Omega$,
  $z_0$ and $\varepsilon_0$
  Most notably the constants are independent of $z$, $\varepsilon$, and $\mathcal{N}_{\Omega}$.  
\end{assumption}

For simplicity of notation, we assume that the constants $b$ and $\mu$ in Assumptions~\ref{ass:approx_of_unitial_condtiion} and~\ref{ass:approx_of_HHx}
coincide. All our results will hold for general spaces $\HHx$, as long as they resolve specific scales. We will later provide
a concrete example of constructing such spaces in $1D$ and 2D, see also~\cite{tensor_fem}.

%

The next lemma collects some facts about the time evolution. These results are well-known for the
case of the heat equation, and their proof easily carries over to our setting.
\begin{lemma}
\label{lemma:apriori}
The following statements hold for the continuous and the semidiscrete problems:
  \begin{enumerate}[(i)]
    \item
      \label{it:sg_is_l2_continuous}
      The maps $t \mapsto u(t)$ and $t \mapsto u_h(t)$ are in $C\big([0,\infty),L^2(\Omega)\big)$.
    \item
      \label{it:sg_apriori_stability}
      For all $t>0$ and $\ell \in \N_0$, $\beta \in (0,1)$ and $\gamma \in [0,1]$ such that $2\ell + \gamma - \beta \geq 0$,
      \begin{align*}
        \norm{\sg(t) u_0}_{L^2(\Omega)} &\leq \norm{u_0}_{L^2(\Omega)} 
                                        \quad \text{ and } \quad 
        \norm{{\left[\sg(\cdot) u_0\right]}^{(\ell)}(t)}_{\widetilde{H}^{s\gamma}(\Omega)}
                                        \! \lesssim t^{-\ell + \frac{\beta-\gamma}{2}} \norm{u_0}_{\widetilde{H}^{s\beta}(\Omega)},         
      \end{align*}
      provided that the right-hand side is finite.
      In the discrete setting, these estimates read as
      \begin{align*}
        \norm{\sg_h(t) u_0}_{L^2(\Omega)} &\leq \norm{u_0}_{L^2(\Omega)} 
                                        \quad \text{ and } \quad 
        \norm{{\left[\sg_h(\cdot) u_0\right]}^{(\ell)}(t)}_{\widetilde{H}^{s\gamma}(\Omega)}
                                        \! \lesssim t^{-\ell + \frac{\beta-\gamma}{2}} \norm{u_0}_{\mathbb{V}^{\mathcal{X}}_{h,\beta}},         
      \end{align*}
      provided that the right-hand side is finite.

    \item  \label{it:sg_apriori_stability_inh1}
      Set $w_h:=\int_{0}^{t}{\sg_h(\tau) \Pi_{L^2} f(t-\tau) \,d\tau}$. Then the following estimates hold:
      \begin{align*}
        \norm{w_h(t)}_{L^2(\Omega)}^2&\lesssim t \int_{0}^{t}{\!\norm{\Pi_{L^2}f(\tau)}_{L^2(\Omega)}^2 d\tau}
        \,\; \text{and} \;\,
        \int_{0}^{t}{\!\norm{\dot{w}_h(\tau)}_{L^2(\Omega)}^2 \,d\tau}\lesssim \int_{0}^{t}{\!\norm{\Pi_{L^2}f(\tau)}_{L^2(\Omega)}^2 d\tau}.
      \end{align*}
    \end{enumerate}
\end{lemma}
\begin{proof}  
      
    Statement~(\ref{it:sg_is_l2_continuous}) is one of the defining properties of a $C_0$-semigroup; thus it follows from Theorem~\ref{thm:llh_generates_sg}.
    The statements of~(\ref{it:sg_apriori_stability}) follow by considering the eigen decomposition of $u$, see~\cite[Lemma 3.2]{thomee_book}.
    The proof of~(\ref{it:sg_apriori_stability_inh1}) is an energy argument.
\end{proof}

Corresponding to the operator $\LL^s_h$, we define the Ritz approximation $\Pi_h: \dom(\LL^s) \to \HHx$ via
\begin{align}
  \label{eq:def_ritz}
  \ltwoprodX{\LL_h^s \Pi_h u}{v}=\ltwoprodX{\LL^s u}{v} \qquad \forall v \in \HHx.
\end{align}
(Note: unlike in the heat equation case, the operator $\Pi_h$ is not a projection).
Since the bilinear form on the left-hand side is elliptic by Lemma~\ref{lemma_llh_is_elliptic}
and $\ltwoprodX{\LL^s u}{v}$ is a linear functional in $v$, $\Pi_h u$ exists and is well defined.
(Since $\HHx$ is finite dimensional we do not have to worry about the  norms involved.)

\begin{lemma}
  Let $u$ solve~\eqref{eq:model_problem}, and $u_h$ solve~\eqref{eq:sd_semigroup}. Define $\rho:=u - \Pi_h u$ and
  $\theta:= \Pi_h u - u_h$. 
  Then  $\theta$ satisfies the following  semidiscrete equation for all $t>0$:
\begin{align}
  \label{eq:sd_ode_for_theta}
  \dot{\theta}(t) + \LL^s_h \theta(t) &= \dot{\rho}(t), \qquad 
                                        \theta(0)=u_0- u_{h,0}.
\end{align}
\end{lemma}
\begin{proof}
  Straightforward computation, see \cite[Equation (1.27)]{thomee_book}.
\end{proof}

The following proposition holds:
\begin{proposition}
  \label{prop:error_up_to_ritz}
  Let $u$ solve~\eqref{eq:model_problem}, and $u_h$ solve~\eqref{eq:sd_semigroup}. Define $\rho:=u - \Pi_h u$ and
  $\theta:= \Pi_h u - u_h$. Then the following estimates hold for all $t>0$:
  \begin{align}
    \int_{0}^{t}{\norm{\theta(\tau)}^2_{L^2(\Omega)} \,d\tau}
    &\lesssim t \norm{u_0 - u_{h,0}}_{L^2(\Omega)}^2 +  \int_{0}^{t}{\norm{\rho(\tau)}^2_{L^2(\tau)} \,d\tau}, 
    \label{eq:error_up_to_ritz_l2}  \\
    t \norm{\theta(t)}_{L^2(\Omega)}^2 + \int_{0}^{t}{\tau \norm{\theta(\tau)}_{\widetilde{H}^{s}(\Omega)}^2 \,d\tau} 
    &\lesssim
      \begin{multlined}[t]
        t \norm{u_0 - u_{h,0}}_{L^2(\Omega)}^2 
        +\int_{0}^{t}{\tau^2 \norm{\dot{\rho}(\tau)}^2_{L^2(\Omega)} + \norm{\rho}_{L^2(\Omega)}^2 \,d\tau} \\
        + \sup_{\tau \in (0,t)}\left( \tau \norm{\rho(\tau)}_{L^2(\Omega)}^2\right),
        \end{multlined}
  \label{eq:error_up_to_ritz_pointwise_and_hone}
\end{align}
\end{proposition}
\begin{proof}
  These estimates are well known for the case of the heat equation. Similar results and techniques 
  can be found, for example, in~\cite[Chapter 3]{thomee_book}. The use of the backward parabolic 
  problem goes back at least to~\cite{luskin_rannacher_smoothing}.   
  For completeness, we provide a proof in Appendix~\ref{sect:proof_error_up_to_ritz}.
\end{proof}

The previous results  mean that it is sufficient to analyze the behavior of the Ritz approximation when applied to $u$. 
We start this endeavor by  showing that the Ritz approximation is quasi-optimal.
\begin{lemma}
  \label{lemma:ritz_is_quasioptimal}
  Let $u \in \dom(\LL^s)$, and let $\lifting u$ denotes its lifting to $\HH$ defined in~\eqref{eq:continuous_lifting}. Then the following estimate holds:
  \begin{align*}
    \norm{u - \Pi_h u}_{\widetilde{H}^s(\Omega)} \lesssim \inf_{ \VV_h \in \HHxy}  \norm{ \lifting u - \VV_h}_{\HH}.
  \end{align*}
\end{lemma}
\begin{proof}
  We set $u_h:=\Pi_h u$,  and show Galerkin orthogonality 
  $\AA(\lifting u - \lifting_h u_h, \VV_h)=0$ for all $\VV_h \in \HHxy$.
  We first note that $\AA(\lifting u,\VV_h)$ and $\AA(\lifting_h u_h ,\VV_h)$ depend only on the trace of $\VV_h$:
  By the definition of the liftings
    (see~\eqref{eq:discrete_lifting_pde} and~\eqref{eq:continuous_lifting_pde} respectively), we have for $\WW_h \in \HHxy$ with $\trace{\WW_h}=\trace{\VV_h}$:
    \begin{align*}
      \AA(\lifting u,\VV_h-\WW_h)&= 0 \qquad \text{ and } \qquad  \AA(\lifting_h u_h, \VV_h-\WW_h)=0.
    \end{align*}    
    Therefore, we get by inserting the definition of $u_h=\Pi_h u$ and~\eqref{eq:def_ritz}:
    \begin{align*}
      \AA(\lifting u - \lifting_h u_h, \VV_h)
      &=\AA(\lifting u - \lifting_h u_h, \lifting_h \trace{\VV_h}) \\ 
      &=\AA(\lifting u, \lifting_h \trace{\VV_h})  - \AA(\lifting u, \lifting \trace{\VV_h})
      =0,
    \end{align*}
 since $\trace{\left( \lifting_h \trace{\VV_h} - \lifting \trace{\VV_h}\right)} = 0$ and 
 (\ref{eq:continuous_lifting}) holds. 
  The approximation result then follows easily from the boundedness of the trace operator and the ellipticity of $\AA$.
\end{proof}

The combination of Proposition~\ref{prop:error_up_to_ritz} and Lemma~\ref{lemma:ritz_is_quasioptimal}
shows that we need to study the best approximation of $\UU(t)=\lifting[u(t)]$ in the space $\HHxy$. 
This will be done in the next sections. 

\subsection{A related elliptic problem}
\label{section:elliptic_problem}
In this section, we analyze a family of elliptic problems that will allow us to pass from the 
function $u\in \widetilde{H}^s(\Omega)$ to $\mathscr{U} \in \HH$.

  Instead of using the more intuitive lifting $\lifting_h$, we use one in the form of
  a Neumann problem. This is done so as to be able to reuse the techniques developed in~\cite{tensor_fem}
  instead of having to analyze a Dirichlet problem from scratch.

\begin{definition}
  Let $\lambda > 0$ be fixed. 
  For $f \in L^2(\Omega)$, we define the solution operator $\solve^\lambda f$ by:
\begin{subequations}
\label{eq:G_lambda}
  \begin{align}
    -\fdiv(y^\alpha A \nabla \solve^\lambda f ) + y^\alpha c \solve^\lambda f &= 0  \quad &&\text{in $\CC$}, \\ 
    d_s \lambda \trace{\solve^\lambda f} + \partial_\nu^{\alpha}{\solve^\lambda f} &= d_s f  \quad&& \text{on $\Omega \times \{0\}$},\\
    \solve^\lambda f &= 0 \quad &&\text{on } \partial_L \CC.
  \end{align}
\end{subequations}
\end{definition}

\begin{lemma}
  \label{lemma:elliptic_solve_stability}
  The following stability estimate holds:
  \begin{align}
    \label{eq:elliptic_solve_stability}
    \norm{\solve^\lambda f}_{\HH}&\lesssim \lambda^{-1/2} \norm{f}_{L^2(\Omega)}.
  \end{align}
  The implied constant depends only on $c$, $A$, and $\Omega$ but is independent of $\lambda$ and $f$.
\end{lemma}
\begin{proof}
  We note that
  $$\norm{\solve^\lambda f}_{\HH}^2 \lesssim \AA(\solve^\lambda f,\solve^\lambda f)
  \lesssim \AA(\solve^\lambda f,\solve^\lambda f) + d_s \lambda \ltwoprodX{\trace{\solve^\lambda f}}{\trace{\solve^\lambda f}}.$$
  
  Inserting the definition of $\solve^\lambda$ gives:
  \begin{align*}
    \AA(\solve^\lambda f,\solve^\lambda f) &+ d_s \lambda \ltwoprodX{\trace{\solve^\lambda f}}{\trace{\solve^\lambda f}}
    =d_s\ltwoprodX{f}{\trace{\solve^\lambda f}} \\ 
    &\begin{aligned}[t]      
      &\lesssim \lambda^{-1/2} \norm{f}_{L^2(\Omega)}
      \left[ \AA(\solve^\lambda f,\solve^\lambda f) + \lambda d_s \ltwoprodX{\trace{\solve^\lambda f}}{\trace{\solve^\lambda f}}\right]^{1/2}.
    \end{aligned}\qedhere
  \end{align*}
\end{proof}

\begin{remark}
  This ``damping property'' of the factor $\lambda^{-1/2}$ in~\eqref{eq:elliptic_solve_stability}
  is the main motivation for considering such operators, compared to the more intuitive
  $\lambda=0$ case, which is the operator analyzed in~\cite{tensor_fem}.
  It will allow us to better control the behavior of $\UU$ for small times $t$ by choosing $\lambda \sim 1/t$, see Section~\ref{sect:semidiscretization_2}.
  It is also the operator which needs to be inverted when discretizing using a implicit Euler timestepping scheme, where $\lambda^{-1}$
  is the timestep size, see Section~\ref{sect:time_discretization}. We also point out the 
  strong relation of the operator $\solve^\lambda$ to the resolvent $(\lambda + \LL)^{-1}$, see the proof of Theorem~\ref{thm:llh_generates_sg}.
\eremk
\end{remark}

\subsubsection{Discretization of the extended variable $y$}
\label{sect:disc_of_y}
\paragraph{$hp$-fem in 1d:}
In this section, we introduce the basics of $hp$-Finite Elements in 1D. This will provide us with the discretization scheme
for the extended variable $y$. Additionally, it will serve as a model construction for satisfying Assumptions~\ref{ass:approx_of_unitial_condtiion} and
\ref{ass:approx_of_HHx}.

We introduce the notion of a geometrically refined mesh. 
For a grading factor $0<\sigma<1$  and $L \in \N$ layers, the geometric mesh on the domain $(-1,1)$
refined towards $-1$, denoted by  $\TT^L_{(\mathbf{-1},1)}:=(x_i)_{i=0}^{L+1}$ is given by
\begin{align*}
  x_0:=-1, \; x_i:=-1 + \sigma^{L - i +1}, \; i=1,\dots L, \; x_{L+1}:=1.
\end{align*}

Analogously we define the geometric mesh refined towards $1$ and denote it by $\TT^L_{(-1,\mathbf{1})}$,
and the mesh geometrically refined towards both endpoints $\TT_{(\mathbf{-1},\mathbf{1})}^L$ with nodes
at 
\begin{align*}
  x_0&:=-1, \; x_i:=-1 + \sigma^{L - i +1}, \; i=1,\dots L, \;\\
  x_i&:=1-\sigma^{i - L}, \; i=L+1,\dots 2L, \;\quad  x_{2L+1}:=1.
\end{align*}

In general, triangulations on $(a,b)$, for example denoted by $\TT_{(\mathbf{a},b)}^L$ are obtained by an affine mapping of $\TT_{(\mathbf{-1},1)}^L$ etc.

Let $\TT$ be a triangulation of a domain $\Omega$. 
For a polynomial degree distribution $\mathbf{r} \in \N_0^{\abs{\TT}}$, we define the space of piecewise polynomials
\begin{align*}
  \SS^{\mathbf{r},1}(\TT):=\left\{ u \in C(\Omega): \, u|_{K_i} \text{ is a polynomial of degree $\mathbf{r}_i$ } \;\forall K_i \in \TT\right\}.
\end{align*}
For the discontinuous case, we define:
\begin{align*}
  \SS^{\mathbf{r},0}(\TT):=\left\{ u: \Omega \to \R, \,  u|_{K_i} \text{ is a polynomial of degree $\mathbf{r}_i$ } \;\forall K_i \in \TT \right\}.
\end{align*}

To simplify the notation, we write $\SS^{p,1}(\TT):=\SS^{(p,\dots,p),1}(\TT)$ for the case of constant polynomial degree $p$,
and analogously for $\SS^{p,0}(\TT)$.

We also sometimes need to impose Dirichlet conditions on the boundary. We write 
$$\SS^{\mathbf{r},1}_{0}(\TT):=\{ u \in \SS^{\mathbf{r},1}(\TT): u|_{\partial \Omega} =0\}.$$

\paragraph{The space $\HHy$:}
We now give the precise construction for the space $\HHy$. It is based on an $hp$-FEM on a graded mesh.
The details are laid out in the next definition.
\begin{definition}
\label{def:geo-mesh-y}
  Fix $\YY > 0$. Let $\TT_{(\mathbf{0},\YY)}^L$ be a geometric mesh on $(0, \YY)$, refined towards $0$ with $L$ layers  and a grading factor $\sigma \in (0,1)$, i.e., 
given by the nodes $\{0,\mathcal{Y}\,\sigma^i\,|\, i=0,\ldots,L\}$. Assume that $\YY \sim L$.
  Let $\HHy:=\SS^{\mathbf{r},1}(\TT^L_{(\mathbf{0},\mathcal{Y})}) \cap \{u: u(\YY)=0\}$ be the space of piecewise polynomials with degree distribution
  vector $\mathbf{r}$ which vanish at the endpoint $\YY$.
\end{definition}

Using the eigenpairs $(\varphi_j,\mu_j)_{j=0}^\infty$ from~\eqref{eq:def_fractional_laplacian},
we have the following representation of $\UU:=\solve^{\lambda} f$:
\begin{align*}
  \UU(x,y)&=\sum_{j=0}^{\infty}{u_j \varphi_j(x) \psi_j(y)} \qquad \text{ with }\qquad 
  u_j:=\left(\lambda + \mu_j^s\right)^{-1} \ltwoprodX{f}{\varphi_j}.           
\end{align*}
Here $\psi_j$ are the functions from \cite[Formula (4.2)]{tensor_fem}. They satisfy the differential equation:
\begin{align*}
  \psi_j'' +  \frac{\alpha}{y} \psi_j' - \mu_j \psi_j = 0 \qquad& \text{in } (0, \infty), \\
  \psi_j(0)=1,  \qquad   &\lim_{y\to \infty}\psi_j(y) = 0.
\end{align*}

\begin{lemma}
  \label{lemma:apriri_of_fourier_coefficients}
  The coefficients  $u_j$ satisfy the follwing \textsl{a priori} estimate:
  \begin{align*}
     \abs{u_j}^2 &\lesssim \lambda^{-2}  \, |\!\ltwoprodX{f}{\varphi_j}\!|^2 
    \quad \text{ and } \quad 
     \mu_j^s \abs{u_j}^2\lesssim \lambda^{-1} \, |\!\ltwoprodX{f}{\varphi_j}\!|^2.
  \end{align*}
\end{lemma}
\begin{proof}
  From the definition, we get by multiplying with $u_j$:
  \begin{align*}
    \lambda \abs{u_j}^2 + \mu_i^s \abs{u_j}^2 &=\ltwoprodX{f}{\varphi_j} u_j
                                    \leq |\!\ltwoprodX{f}{\varphi_j} \!| \abs{u_j},
  \end{align*}
  which implies $\lambda \abs{u_j} \leq |\!\ltwoprodX{f}{\varphi_j}\!|$. Inserting this knowledge
  gives:
  \begin{align*}
    \mu_i^s \abs{u_j}^2 &\leq \lambda^{-1} |\!\ltwoprodX{f}{\varphi_j} \!| \lambda \abs{u_j}   \leq \lambda^{-1} |\!\ltwoprodX{f}{\varphi_j}^2. \qedhere
  \end{align*}
\end{proof}

\begin{lemma}
  \label{lemma:approx_in_y}  
  Let $\Pi_{\YY}$ denote the Galerkin projection onto the space $H_0^1(\Omega)\otimes \HHy$
for the problem (\ref{eq:G_lambda}).
Then the following estimate holds
  for all $f \in L^2(\Omega)$:
  \begin{align*}
    \norm{\solve^\lambda f - \Pi_{\YY} \solve^\lambda f }_{\HH}\lesssim \lambda^{-1/2} e^{-b p } \norm{f}_{L^2(\Omega)}.
  \end{align*}
\end{lemma}
\begin{proof}
  We follow the argument of \cite[Secs.~4, 5]{tensor_fem}. By Galerkin orthogonality, we are only concerned with proving an estimate for the best approximation to $\solve^\lambda f$.
  The functions $\psi_j$ all decay exponentially for $y \to \infty$.  We can bound
  \begin{align*}
    \norm{\solve^\lambda f (y) - \VV^{\text{cutoff}}}_{\HH}\lesssim C e^{-\sqrt{\lambda_1} \YY/4} \sqrt{\sum_{j=0}{\mu_j^s \abs{u_j}^2}}
    &\lesssim C e^{-\sqrt{\lambda_1} \YY/4} \lambda^{-1/2} \norm{f}_{L^2(\Omega)}, 
\end{align*}
 where 
  $\lambda_1 > 0$ denotes the smallest eigenvalue of the operator $\LL$ on $\Omega$, see \cite[Lemma 3.3]{pde_approach_apriori} for details;
  the proof can be taken verbatim, just replacing the definition of the coefficients $u_j$.
  It is thus sufficient to study the approximation on the finite cylinder $\Omega \times (0, \YY)$.

  We define the weights $\omega_{\beta,\gamma}:=y^\beta e^{\gamma y}$, and the weighted $L^2$-norms
  \begin{align*}
    \norm{v}_{L^2(\omega_{\beta,\gamma },\CC)}^2
    := \int_{0}^{\infty}{\int_{\Omega}{\omega_{\beta,\gamma}(y)\abs{v(x,y)}^2 \,dx \, dy}}.
  \end{align*}
  We note that the function $\solve^\lambda u$ satisfies the following \textsl{a priori} estimates:
  \begin{align*}
    \norm{\partial_y^{\ell +1} \solve^\lambda f}_{L^2(\omega_{\alpha +2 \ell },\CC)}
    &\lesssim \lambda^{-1/2} \kappa^{\ell +1 }(\ell +1 )! \norm{f}_{L^2(\Omega)} \quad \forall \ell \in \N_0,\\
    \norm{ \nabla_{x} \partial_y^{\ell +1} \solve^\lambda f}_{L^2(\omega_{\alpha + 2 (\ell +1 )},\CC)}
    &\lesssim \lambda^{-1/2} \kappa^{\ell +1 }(\ell +1 )! \norm{f}_{L^2(\Omega)} \quad \forall \ell \in \N_0.
  \end{align*}
  Again, this follows~\cite[Theorem 4.7]{tensor_fem} verbatim, only plugging in the stronger estimate for the coefficients $u_j$ to 
  gain the extra factor $\lambda^{-1/2}$. 
  This in turn implies that $\solve^{\lambda} f$ is in some Banach-space valued countably normed spaces.
  Invoking the interpolation operator $\Pi_{y,\{\YY\}}^{\mathbf{r}}$ from \cite[Section 5.5.1]{tensor_fem} then shows the
  stated result.
\end{proof}

\subsubsection{Discretization in $x$}
\label{sect:disc_of_x}
In this section, we study the discretization error due to the choice of space $\HHx$.
We will show that the requirement that $\HHx$ resolves appropriate scales (see Assumption~\ref{ass:approx_of_HHx}) suffices
to show exponential convergence.

Before we prove an approximation result for $\solve^\lambda$, we need the following result on 
the solution of singularly perturbed problems, generalizing the theory developed in, e.g., 
\cite{melenk97,melenk_book} (for real singular perturbation parameters)
to the case where the right hand side is itself the solution to a singularly perturbed problem: 
\begin{lemma}
  \label{lemma:approx_of_iterated_solve}
  Let $\varepsilon > 0$ and $z \in \mathscr{S}$ with $\Re(z) \geq 0$.
  Assume that the space $\HHx$ resolves
  the scales $\varepsilon$ and $\abs{z}^{-1/2}$, as defined in Assumption~\ref{ass:approx_of_HHx}. 
  Let $u_{z} \in H_0^1(\Omega)$ denote the solution to $(\LL-z) u_z = z\,f$, where $f\in L^2(\Omega)$ is analytic on $\overline{\Omega}$.
  Let $u \in H_0^1(\Omega)$ solve
  \begin{align}
    \label{eq:iterated_solve_eqn}
    \varepsilon^2 \LL  u  + u = u_{z}.
  \end{align}

  Then the following best approximation result holds:
  \begin{align*}
    \inf_{v_h \in \HHx}\left[\varepsilon^2 \norm{\nabla u - \nabla v_h}^2_{L^2(\Omega)} + \norm{u-v_h}^2_{L^2(\Omega)}\right]\leq C e^{-b \mathcal{N}_{\Omega}^\mu}.
  \end{align*}
  The constant $C$ depends on $\mathscr{S}$, the constants of analyticity of $f$,
  and the  constants from Assumption~\ref{ass:approx_of_HHx} but not on $\varepsilon$ or $z$.
  $b$ and $\mu$ are as in Assumption~\ref{ass:approx_of_HHx}.
\end{lemma}
\begin{proof}
  We make the ansatz $u=\nu u_z - w$, for $\nu \in \C$ and some function $w \in H_0^1(\Omega)$.
  Plugging  this decomposition  into \eqref{eq:iterated_solve_eqn} and using the PDE for $u_z$,
  we get $\nu=\frac{1}{1+\varepsilon^2 z}$ and that 
  $w$ solves $$\varepsilon^2 \LL  w + w = \frac{\varepsilon^2 z}{1+\varepsilon^2 z}  \,f.$$
  Since we assumed $\Re(z)\geq 0$, the coefficient $\nu$ is 
  bounded independently of $\varepsilon$ and $z$.
  We also compute 
  $$\abs{1+\varepsilon^2 z}^2=(1+\varepsilon^2 \Re(z))^2 + \varepsilon^4\Im(z)^2 > \varepsilon^4\abs{z}^2,$$
  which shows that $\frac{\varepsilon^2 z}{1+\varepsilon^2 z}$ is also uniformly bounded.
  
  Since we assumed that the mesh resolves the scale $\varepsilon$, we can apply Assumption~\ref{ass:approx_of_HHx} to
   $w$ to get the estimate:
  \begin{align*}
    \inf_{v_h \in \HHx}\left[\varepsilon^2 \norm{\nabla w - \nabla v_h}^2_{L^2(\Omega)} + \norm{w-v_h}^2_{L^2(\Omega)}\right]\leq C e^{-b \mathcal{N}_{\Omega}^\mu}.
  \end{align*}
  
  We also assumed that the mesh resolves the scale $\abs{z}^{-1/2}$. Thus we get an exponential approximation property
  for $u_z$ in the $\abs{z}^{-1/2}$ weighted norm.  
  In order to get the estimate in the $\varepsilon$-weighted norm, we note that for $\varepsilon < \abs{z}^{-1/2}$ we get the estimate trivially.
  For $\varepsilon > \abs{z}^{-1/2}$ we note that
  \begin{align*}
    \varepsilon^2 \norm{\nu\nabla u_z}_{L^2(\Omega)}^2
    &= \abs{\frac{\varepsilon^2 z}{1+\varepsilon^2 z}} \abs{z}^{-1} \norm{\nabla u_z}_{L^2(\Omega)}^2
   \lesssim \abs{z}^{-1}\norm{\nabla u_z}_{L^2(\Omega)}^2.
  \end{align*}
  This means we can approximate $\nu\, u_z$ in the $\varepsilon$-weighted norm at an exponential rate, which concludes the proof.
\end{proof}

We now employ the decoupling strategy of \cite{tensor_fem}.
Let $(v_i)_{i=0}^{\mathcal{M}} \subseteq \HHy$ denote a basis with the following properties:
\begin{align}
\label{eq:EWP}
  d_s \lambda v_i(0) v_j(0)  + \int_{0}^{\YY}{y^\alpha \,v_i' v_j' } = \delta_{ij} 
  \qquad &\text{ and } \qquad
           \int_{0}^{\YY}{y^\alpha \,v_i v_j } = \kappa_i \delta_{ij},
\end{align}
for coefficients $\kappa_i>0$. Since the bilinear forms are SPD, such a basis exists.
On $\Omega$, we define the bilinear forms
\begin{align}
  \label{eq:def_of_a_kappa}
  a_{\kappa_i}(U,V):=\kappa_i  \left[\ltwoprodX{\nabla U}{ \nabla V} +  c\ltwoprodX{U}{ V}\right] + \ltwoprodX{U}{V},
\end{align}
and note that the following norm equivalence holds on $H_0^1(\Omega)\otimes \HHy$ for all 
$\VV:=\sum_{i=0}^{\mathcal{M}}{\VV_i v_i}$:
\begin{align}
  \label{eq:norm_equivalence}
  \lambda \norm{\trace{\VV}}_{L^2(\Omega)}^2 +\norm{\VV}_{\HH}^2
  &\sim\sum_{j=0}^{\mathcal{M}}{a_{\kappa_i}(\VV_i,\VV_i)}.
\end{align}
  \eqref{eq:norm_equivalence} shows that estimates
in the $\HH$ norm can also be obtained from bounds on each component in the corresponding $\kappa_i$-weighted $H^1$-norm.

The bilinear forms $a_{\kappa_i}$ correspond to singularly perturbed problems for small $\kappa_i$. 
We want to apply Assumption~\ref{ass:approx_of_HHx}.
For this we need bounds on the $\kappa_i$ as well as on $v_i(0)$.
\begin{lemma}
\label{lemma:EWP-estimates}
  Let $h_{\operatorname{min}}> 0 $ denote the smallest element size  in $\TT^L_{(\mathbf{0},\mathcal{Y})}$ and $p$ the maximal polynomial degree used
  for $\HHy$.
  Then the eigenpairs $(v_i,\kappa_i)_{i=0}^{\mathcal{M}}$ of (\ref{eq:EWP}) satisfy: 
  \begin{align} 
    \frac{h_{\operatorname{min}}^2}{p^4(1+\lambda \YY^{1-\alpha})}  &\leq \kappa_i \leq C \YY^2(1- \alpha^2)^{-1} \label{eq:estimate_kappa}, \\    
    \abs{v_i(0)} &\leq  \lambda^{-1/2}. \label{eq:estimate_traces_of_vi}
  \end{align}
    
\end{lemma}
\begin{proof}
  By definition we have  $1=d_s \lambda v_i(0)^2  + \int_{0}^{\YY}{y^\alpha |v_i'|^2 }= \kappa_i^{-1} \int_{0}^{\YY}{y^{\alpha}|v_i|^2}$,
  or $\kappa_i= \int_{0}^{\YY}{y^{\alpha}|v_i|^2}$.
  By the proof of \cite[Lemma B.2]{tensor_fem} we can estimate 
  $$\kappa_i = \norm{v_i}^2_{L^2(y^\alpha,(0,\YY))} \lesssim \YY^2(1-\alpha^2)^{-1}  \norm{v_i'}^2_{L^2(y^\alpha,(0,\YY))} \lesssim \YY^2(1 -\alpha^2)^{-1}.$$
\cite[Lemma~{B.1}]{tensor_fem} provides $|v_i(0)|^2 \leq \YY^{1-\alpha}/(1-\alpha^2) \norm{v_i'}^2_{L^2(y^\alpha,(0,\YY))}$.
This and the inverse estimate from \cite[Lemma~{B.3}]{tensor_fem} yield
%
  \begin{align*}
\kappa_i^{-1} \norm{v_i}^2_{L^2(y^\alpha,(0,\YY))} &= \lambda d_s \abs{v_i(0)}^2 + \norm{v_i'}_{L^2(y^\alpha,(0,\YY))}^2
    \lesssim (1+\lambda \YY^{1-\alpha})\norm{v_i'}_{L^2(y^\alpha,(0,\YY))}^2  \\
    &\lesssim h_{\operatorname{min}}^{-2} p^4  (1 + \lambda \YY^{1-\alpha}) \norm{v_i}_{L^2(y^\alpha,(0,\YY))}^2.
  \end{align*}
  
  To see~\eqref{eq:estimate_traces_of_vi}, we calculate:
  \begin{align*}
    \abs{v_i(0)}^2&\leq \lambda^{-1} d_s^{-1} \left[\lambda d_s \abs{v_i(0)}^2 + \norm{v_i'}_{L^2(y^\alpha,(0,\YY))}^2\right] = \lambda^{-1}.  \qedhere
  \end{align*}
\end{proof}

\begin{lemma}
  \label{lemma:best_approximation_of_elliptic_part}
  Let $u\in L^2(\Omega)$ be either holomorphic in $\overline{\Omega}$ or the solution to the singularly perturbed problem
  $-z^{-1} \LL u + u = f$ with $f$ holomorphic  on $\overline{\Omega}$ and $z\in \mathscr{S}$ with $\Re(z)\geq 0$. 
  Assume that $\HHx$ resolves the scales ${\abs{z}}^{-1/2}$ and $\sqrt{\kappa_i}$ for all $i=0,\dots,\mathcal{M}$.

  Then the following best approximation result holds:
  \begin{align*}
    \inf_{\VV_h \in \HHxy}\norm{\solve^\lambda u - \VV_h}_{\HH} \lesssim \lambda^{-1/2} \left( e^{- b \mathcal{N}_{\Omega}^{\mu}} + e^{- b \sqrt{\mathcal{N}_{\YY}}}\right).
  \end{align*}
  where $\mu$ is the exponent for $\HHx$ in Assumption~\ref{thm:hp_fem_resolves_scales}. The constant $b$ depends
    on the domain of holomorphy of $u$ or $f$.
\end{lemma}
\begin{proof}  
  By Lemma~\ref{lemma:approx_in_y}, it is sufficient to consider a semidiscrete functions 
$\UUdy:=\Pi_{\mathcal{Y}}\solve^\lambda u \in H_0^1(\Omega)\otimes \HHy$ and their
  approximation in $\HHxy$.   
  Using the basis $(v_j)_{j=0}^{\mathcal{M}}$, the function $\UUdy=:\sum_{i=0}^{\mathcal{M}}{U_i v_i}$ 
  from Lemma~\ref{lemma:approx_in_y} solves:
  \begin{align*}
    a_{\kappa_i}(U_i,V)&=d_s v_i(0) \ltwoprodX{u}{V} \qquad \forall V \in H_0^1(\Omega).
  \end{align*}

  This is just the weak formulation of the singularly perturbed problems from Assumption~\ref{ass:approx_of_HHx},
  with $\varepsilon=\sqrt{\kappa_i}$. Note that $|v_i(0)| \lesssim \lambda^{-1/2}$ by (\ref{eq:estimate_traces_of_vi}).
Since we assumed that the scales are resolved, 
  we can either apply Assumption~\ref{ass:approx_of_HHx} directly (if $u$ is holomorphic on $\overline{\Omega}$) or 
  apply Lemma~\ref{lemma:approx_of_iterated_solve} (if $u$ solves $-z^{-1} \LL u + u =f$) to get the following estimate for 
  the best approximations $\Pi U_i \in \HHx$:
  \begin{align*}
    \kappa_i \norm{ \nabla [U_i - \Pi U_i]}_{L^2(\Omega)}^2 + \norm{U_i - \Pi U_i}^2_{L^2(\Omega)}   
    &\lesssim C(u)\lambda^{-1} e^{-b \mathcal{N}_{\Omega}^{\mu}},
  \end{align*}
  the norm equivalence~\eqref{eq:norm_equivalence} then concludes the proof.
\end{proof}

\subsection{Returning to the semidiscretization}
\label{sect:semidiscretization_2}

We are now in a position to show exponential convergence for the best approximation (and thus also the Ritz approximation) 
of the exact solution $\UU$. We first consider positive times $t$ bounded away from $0$.  In this regime,
our finite element mesh is assumed to resolve the pertinent scales. The smaller times, for which the scales
are not resolved, are treated separately later on.
\begin{theorem}
\label{thm:best_approximation_HHx}
Let $t\geq t_0 >0$ be fixed. Let $u_0$ be analytic on a fixed neighborhood $\widetilde{\Omega} \supset \overline{\Omega}$
(but we do not assume boundary conditions, i.e., $u_0 \notin \widetilde{H}^s(\Omega)$ is allowed),
and assume homogeneous right-hand side, i.e.,  $f=0$.
Also assume that for a chosen ``high frequency'' cutoff $\zhf> z_0 >0$, the space $\HHx$ resolves  down to the scale
\begin{align}
    \label{eq:scales_that_need_resolving}
     \varepsilon_{\min}=\min{\left(\sqrt{t_0}\frac{h_{\operatorname{min}}}{p^2} c_\YY,\abs{\zhf}^{-1/2}\right)},
\qquad c_\YY = \sqrt{\frac{1}{t_0 + \YY^{1-\alpha}}}, 
  \end{align}
  where $h_{\min}$ and $p$ are the minimum element size and maximal polynomial degree of $\HHy$.
Then, for each $\ell \in \N_0$,  there exists a function $\VV_h(t) \in \HHxy$ such that the following estimate holds:
\begin{align}
  \norm{ \UU^{(\ell)}(t) - \VV_h(t)}_{\HH}
  &\lesssim t^{-1/2-\ell} \max\left(1,-\log(t)^{1-\min(\ell,1)}\right) \left(
    e^{- b_1 \mathcal{N}_{\Omega}^{\mu}}  + e^{- b_2 \sqrt{\mathcal{N}_{\mathcal{Y}}}} + e^{-\frac{\gamma}{2} {\zhf^s} \, t_0} \right).
                                     \label{eq:best_approximation_HHx_0} 
\end{align}
The implied constant depends on $\Omega$,$\widetilde{\Omega}$, $s$, the constants of analyticity of $u_0$, $z_0$, and the constants from Assumption~\ref{ass:approx_of_HHx}, but is
independent of $t$, $t_0$ and $\zhf$.  The rate $b_2$ also depends on the mesh grading for $y$.
  The rate $b_1$ depends in addition on the constants from Assumption~\ref{ass:approx_of_HHx}. The constant 
$\gamma$ can be chosen to depend on $s$ only. 
\end{theorem}
\begin{proof}
  Since we assumed homogeneous right hand side, we only need to investigate $\UU=\sg(t)u_0$.
  We use the representation of $\sg(t) u_0$ via the Riesz-Dunford calculus 
  (following what is done in \cite[Section 2]{bonito_pasciak_parabolic}) to write:
  \begin{align*}
    \sg(t)u_0 &= \frac{1}{2\pi \ii}\int_{\mathcal{C}}{e^{-t z^{s}} \left( z - \LL \right)^{-1} u_0 \,dz},  \qquad& 
    \left(\sg(t)u_0\right)^{(\ell)} &= \frac{(-1)^\ell}{2\pi \ii}\int_{\mathcal{C}}{ z^{\ell s} e^{-t z^{s}} \left( z - \LL \right)^{-1} u_0 \,dz},
  \end{align*}
  where $\mathcal{C}$ is the following contour consisting of three segments:
  \begin{align*}
    \begin{cases}
      \mathcal{C}_1:=\big\{z(r)=r e^{-\ii \frac{\pi}{4}} & |\quad r\in (r_0, \infty) \big\}\\
      \mathcal{C}_2:=\big\{z(\theta):= r_0 e^{ \ii \theta}\ & |\quad \theta \in (-\pi/4,\pi/4) \big\}\\
      \mathcal{C}_3:=\big\{z(r):=r e^{\ii \frac{\pi}{4} } & |\quad  r\in (r_0, \infty)\big\}\\
    \end{cases}    
  \end{align*}
  and $z^{s}:=e^{ s\log(z)}$ with the logarithm defined with the branch cut along the negative real axis.  
  The parameter $r_0 \in (0, z_0)$  is fixed such that the whole path lies in the domain of ellipticity $\mathscr{S}$,
  as defined in Definition~\ref{def:domain_of_ellipticity}; see Figure~\ref{fig:contour}.

  By adding the term $\frac{1}{ t}  d_s \trace{\UU}$ to both sides of~\eqref{eq:extended_problem_bc}, we
  get that $\UU$ solves
  \begin{alignat*}{3}
    -\fdiv\left(y^\alpha A \nabla \UU \right)+  y^\alpha c \UU &= 0  \qquad &&\text{on $\CC\times \R_+$,}  \\
    \frac{d_s}{t} \trace{\UU} + \partial_{\nu}^\alpha \UU &=  \frac{d_s}{t} \trace{\UU} - d_{s} \trace{\dot{\UU}}  \qquad &&\text{on $\omega \times \{0\} \times (0,T)$},\\
    \UU &= 0 \qquad &&\text{ on $\partial_L \CC$}.
  \end{alignat*}
  Using the operator $\solve^{1/t}$,  we can therefore write the function $\UU$ as 
  \begin{align*}
    \UU&=-\solve^{1/t} \trace{\dot{\UU}} +  \frac{1}{t}\solve^{1/t}{\trace{\UU}},
  \end{align*}
  or using the Riesz-Dunford calculus:
  \begin{align*}
    \UU(t)&=\frac{1}{2\pi \ii}\int_{\mathcal{C}}{e^{-t z^{s}} z^{s} \solve^{1/t} \left[ z -  \LL \right]^{-1}  u_0 \;dz}
            + \frac{1}{t} \frac{1}{2\pi \ii}\int_{\mathcal{C}}{e^{-t z^{s}}  \solve^{1/t} \left[ z - \LL \right]^{-1} u_0 \;dz}.
  \end{align*}
  For the derivatives, a similar formula holds:
  \begin{multline*}
    \frac{d^{\ell}}{dt^\ell}{\UU}(t)=\frac{(-1)^{\ell}}{2\pi \ii}\int_{\mathcal{C}}{e^{-t z^{s}} z^{(\ell+1)s} \solve^{1/t} \left[ z - \LL \right]^{-1} u_0 \;dz}
                       + \frac{1}{t} \frac{(-1)^{\ell}}{2\pi \ii}\int_{\mathcal{C}}{e^{-t z^{s}} z^{\ell s} \solve^{1/t} \left[  z-   \LL \right]^{-1} u_0 \;dz}.
  \end{multline*}  
  Hence, we have to study integrals of the form 
  \begin{align} 
    I_m:=\frac{1}{2\pi \ii}\int_{\mathcal{C}}{e^{-t z^{s}} z^{m\,s} \solve^{1/t} \left[ z - \LL \right]^{-1} u_0 \,dz}, 
\qquad m \in \N_0, 
  \end{align}
and their best approximation, paying attention to the dependence on $t$.

  If $\abs{z} \in (\varepsilon_0, \zhf)$, we obtain from Lemma~\ref{lemma:best_approximation_of_elliptic_part} (since
  we require Assumption~\ref{ass:approx_of_HHx} to hold down to scale $\zhf^{-1/2}$)
  \begin{align}
    \label{eq:best_approximation_of_lifted_resolvent}
    \norm{ \solve^{1/t} \left( \id - z^{-1} \LL \right)^{-1} u_0 - \widehat{\VV}_h(z)}_{\HH}
    &\lesssim t^{1/2}  \big( e^{-b_1 \mathcal{N}_{\Omega}^\mu} + e^{-b_2 \sqrt{\mathcal{N}_{\mathcal{Y}}}}\big)
  \end{align}
  for some function $\widehat{\VV}_h(z) \in \HHxy$. If we pick $\widehat{\VV}_h(z)$ as the Galerkin approximation,
    we get continuous dependence on $z$.
  On $\mathcal{C}_2$, we can therefore estimate:
  \begin{align*}
    \norm{\int_{\mathcal{C}_2}{e^{-t z^{s}} z^{m\,s-1} \solve^{1/t} 
    \left[\left( \id - z^{-1} \LL \right)^{-1} u_0 - \widehat{\VV}(z)\right] \,dz}}
    &\leq C t^{1/2} \big( e^{-b_1 \mathcal{N}_{\Omega}^\mu} + e^{-b_2 \sqrt{\mathcal{N}_\mathcal{Y}}}\big).
  \end{align*}
  The more interesting case are the paths $\mathcal{C}_1$ and $\mathcal{C}_2$. We focus on $\mathcal{C}_1$,
  and consider two cases, namely, $\abs{z} \leq \zhf$ and $\abs{z}> \zhf$. 
  In the first case, the mesh resolves the scales down to $\zhf^{-1/2} \leq \abs{z}^{-1/2} $, and we
  can apply Lemma~\ref{lemma:best_approximation_of_elliptic_part}.
  Setting $\gamma:=\cos(\pi\,s/4)$ we estimate:
  \begin{align*}
    I_{m}^1&:=
    \norm{\int_{\mathcal{C}_1 \cap \abs{z} \leq \zhf}{e^{-t z^{s}} z^{m\,s-1}
    \left( \solve^{1/t} \left( z - \LL \right)^{-1} (z u_0) - \widehat{\VV}_h(z) \right)} \,dz}_{\HH} \\
           &\lesssim  t^{1/2} \big( e^{-b_1 \mathcal{N}_{\Omega}^\mu} + e^{-b_2 \sqrt{\mathcal{N}_\mathcal{Y}}}\big)
             \int_{r_0}^{z_{hf}}{e^{- \gamma t r^{s}} r^{m\,s -1} dr}.
  \end{align*}
  Making the substitution $ \gamma t\,r^{s}=:y$, we get:
  \begin{align*}
    \int_{r_0}^{z_{hf}}{e^{-t r^{s}} r^{m\,s-1} dr}    
    &=s^{-1}t^{-m}  \gamma^{-m} \int_{\gamma t r_0^{s}}^{  \gamma t z_{hf}^s}{e^{-y} y^{m-1} \, dr}.
  \end{align*}
  
  We need to consider the case $m=0$ separately, as the integrand then has a singularity at $r=0$. 
  Splitting the integration we get:
  \begin{align*}
    s^{-1}t^{0} \int_{\gamma t r_0^{s}}^{  \gamma t z_{hf}^s}{e^{-y} y^{m-1} \, dr} 
    &\lesssim \int_{\gamma t r_0^{s}}^{1}{e^{-y} y^{-1} \, dr}  +  \int_{1}^{\infty}{e^{-y} y^{-1} \, dr}
    \lesssim \int_{\gamma t r_0^{s}}^{1}{ y^{-1} \, dr}  +  \int_{1}^{\infty}{e^{-y} \, dr}\\
    &\lesssim -\log(\gamma t r_0^{s}) + e^{-1}  \sim 1 - \log(t r_0^s).
  \end{align*}
  
  For $m>0$, we do not get the logarithmic growth for small times, since:
  \begin{align*}
    \gamma^{-m} s^{-1} t^{-m} \int_{y_0}^{y_1}{e^{-y} y^{m-1} \, dr} 
    &\lesssim t^{-m} \int_{0}^{\infty}{e^{-y} y^{m-1}\, dr}
    =t^{-m} \Gamma(m).
  \end{align*}
  Overall, this gives the estimate:
  \begin{align*}
    I_{m}^1 \lesssim  t^{1/2 -m }\max\big(1,-\log(t)^{1-\min(m,1)}\big) \big(e^{-b_1 \mathcal{N}_{\Omega}^\mu}+ e^{-b_2 \sqrt{\mathcal{N}_{\mathcal{Y}}}}\big).
  \end{align*}
  In the case $\abs{z} > \zhf$, we set $\widehat{\VV}_h:=0$ and use the stability estimate~\eqref{eq:elliptic_solve_stability}
  together with the uniform stability of the operator $(z-\LL)^{-1} z$ (see Lemma~\ref{lemma:sing_is_elliptic}).
  For $m>0$,  we estimate:
  \begin{align*}
    \!\!\Big\|\int_{\mathcal{C}_1 \cap \abs{z} > \zhf}{\hspace{-8mm}e^{-t z^{s}} z^{m s-1}\solve^{1/t}\left[z - \LL\right]^{-1} (z  u_0) \,  dz}\Big\|_{\HH} 
      \!\lesssim\! \norm{u_0}_{L^2(\Omega)}  t^{1/2} e^{-\frac{\gamma}{2}  t \,\zhf^{s}}\int_{r > \zhf}{e^{- \gamma t r^{s}/2} r^{m \,s -1} \,dr}  \\
      \lesssim \norm{u_0}_{L^2(\Omega)}   t^{1/2} e^{-\frac{\gamma}{2} t \, \zhf^{s}} \;
      t^{-m}\int_{0}^{\infty}{  e^{-y} y^{m-1} \,dr} 
      \;\lesssim \;\norm{u_0}_{L^2(\Omega)}  e^{-\frac{\gamma}{2}  t \, \zhf^s} \;t^{1/2-m} \;  
      \Gamma(m).    
  \end{align*}
  For $m=0$, the same calculation can be done, but picking up an extra logarithmic term from the integral 
  where $y=z_{hf}^s t \lesssim 1$.
  
  The same argument can be repeated for $\mathcal{C}_3$. The stated estimates then follow easily by setting
  $m=0$ and $m=1$ to estimate $\UU$ (this term involves the logarithmic contributions)
  and $m=\ell$ and $m=\ell+1$ to estimate higher derivatives.
\end{proof}

For small $t<t_0$, we cannot hope to retain exponential convergence, as it would require our mesh to resolve infinitely small scales.
Instead, we rely on on our ability to control the behavior of  the solution near $t=0$ using some smoothness of $u_0$.
\begin{lemma}
  \label{lemma:HHx_approx_small_t0}
  Let $u_0 \in H^{\theta}(\Omega)$ for $0<\theta <1/2$, and assume homogeneous right hand-side, i.e., $f=0$.  
  For all $\ell \in \N_0$,  the following estimate holds for $t >0$:
  \begin{align}
    \norm{ \UU^{(\ell)}(t) }_{\HH}&\lesssim
                                      t^{-\ell -1/2+ \min(\frac{\theta}{2s},1)} \norm{u_0}_{H^{\theta}(\Omega)}.
  \end{align}
  The constant depends on $\Omega$, $\theta$,  $s$ and the coefficients $A$, $c$.
\end{lemma}
\begin{proof}
  For simplicity we assume additionally $\theta\leq 2s$.
  We  note that for $\theta \in (0,1/2)$, the spaces $\widetilde{H}^{\theta}(\Omega)$ and $H^{\theta}(\Omega)$ coincide
  with equivalent norms
  (see~\cite[Section 1.11.6]{triebel3} or~\cite[Theorem~{3.33}, Theorem~{B.9}, Theorem~{3.40}]{McLean2000}).

  Hence, we get $u_0 \in \widetilde{H}^{\theta}(\Omega)$.
  By Lemma~\ref{lemma:apriori}, this implies for $\ell\in \N_0$:
  \begin{align}
    \label{eq:HHx_approx_small_t0_proof1}
    \norm{u^{(\ell)}(t)}_{\widetilde{H}^{s}(\Omega)}\lesssim t^{-\ell+\frac{\theta}{2s}-1/2}\norm{u_0}_{H^{\theta}(\Omega)}.
  \end{align}

  We write $\UU(t)=\lifting u(t)$ using the lifting operator from~\eqref{eq:continuous_lifting}.
  Since the lifting on the continuous level is bounded (see Remark~\ref{remark:continuous_lifting_is_bounded}), we can estimate:
  \begin{align*}
    \norm{\UU^{(\ell)}(t)}_{\HH}
    &= \norm{\lifting u^{(\ell)}(t)}_{\HH}
      \lesssim \norm{u^{(\ell)}(t)}_{\widetilde{H}^{s}(\Omega)}
      \lesssim t^{-\ell + \frac{\theta}{2s} - 1/2}\norm{u_0}_{{H}^{\theta}(\Omega)}.
\qedhere
  \end{align*}
\end{proof}

As a final step before showing convergence of the semidiscrete approximation, we remove the restriction
to homogeneous right-hand sides $f$. This is a simple consequence of the previous results and Duhamel's principle.
\begin{corollary}
\label{cor:best_approx_inhomogeneous_rhs}
Let $t_0 >0$ and $\delta >0$ be fixed. Let $u_0$ be analytic on $\overline{\Omega}$
and assume that $f$ is $\ell$ times continuously differentiable with respect to $t$
such that the functions $f^{(j)}$, $j=0,\ldots,\ell$ are 
uniformly analytic in the sense of Definition~\ref{def:uniformly_analytic}.

Assume that $\HHx$ resolves scales down to  \eqref{eq:scales_that_need_resolving}.
Then, for each $\ell \in \N_0$,  there exists a function $\VV_h(t) \in \HHxy$ such that the following estimates holds
for all $t \in (0,T)$:
\begin{multline}
  \norm{ \UU^{(\ell)}(t) - \VV_h(t)}_{\HH}
  \lesssim  t^{-\ell -1/2} \max\left(1,-\log(t)\right)
    \left(  e^{-b_1 \mathcal{N}_{\Omega}^\mu} + e^{-b_2 \sqrt{\mathcal{N}_\mathcal{Y}}} + e^{-\frac{1}{2} {\zhf^s} \, t_0} \right) \\
    + t_0^{-\ell - 1/2 + \min(\frac{1}{4s}-\delta,1)}.   
    \label{eq:best_approximation_HHx_0_with_rhs}
\end{multline} 
The implied constant depends on the end time $T$, $\Omega$, the data $u_0$, the constants of analyticity of $f^{(j)}$, $\delta$,
and the implied constants in Lemma~\ref{lemma:best_approximation_of_elliptic_part},
e.g., the mesh grading factor. It is independent of $t$, $t_0$,$\zhf$, $\mathcal{N}_{\Omega}$ or $\mathcal{N}_{\mathcal{Y}}$.
For $\ell=0$ and $\ell=1$ we can explicitly give $C(T)\lesssim \max(1,T)$.
\end{corollary}
\begin{proof}
  For $f=0$, this is just a collection of Lemma~\ref{thm:best_approximation_HHx} and~\ref{lemma:HHx_approx_small_t0}.
  For $f\neq 0$ we write
  \begin{align*}
    \UU(t)&=\lifting\left[\sg(t) u_0 + \int_{0}^{t}{\sg(\tau) f(t-\tau) \,d\tau}\right], \\
    \dot{\UU}(t)&=\lifting\left[ (\sg(t) u_0)' + \sg(t) f(0) + \int_{0}^{t}{\sg(t-\tau) \dot{f}(\tau) \,d\tau}\right]
  \end{align*}
  (see \cite[Section 4.2, Corollary 2.5]{pazy} for the derivative of Duhamel's formula).
  The terms involving only $\sg(t)$ are already covered by the results for the homogeneous problem.
  For fixed $\tau \in (0,t)$, the integrand in the last  term corresponds to solving the homogeneous problem with
  initial condition $f(t-\tau)$ (or $\dot{f}(t-\tau)$ in the case of $\dot{\UU}$).
  This means we can also apply Lemmas~\ref{thm:best_approximation_HHx} and~\ref{lemma:HHx_approx_small_t0},
  only picking up an extra power of $t$ due to the additional integration in $\tau$. This gives the stated estimate for $\ell=0$ and $\ell=1$.

  For higher derivatives, we proceed by induction and see that we can  write $\UU^{(\ell)}$  as
  \begin{align*}
    {\UU}^{(\ell)}(t)&=\lifting\left[ (\sg(t) u_0)^{(\ell)} + \sum_{j=0}^{\ell-1}{\left(\frac{d}{dt}\right)^{{\ell-j-1}}[\sg(t) f^{(j)}(0)]}
                       + \int_{0}^{t}{\sg(t-\tau) f^{(\ell)}(\tau) \,d\tau}\right].
  \end{align*}
  All the terms can be estimated as before, where  we estimate $t^{-j} \leq C(T) t^{-\ell}$ and only keep the dominant terms.
\end{proof}

\begin{theorem}
  \label{thm:error_est_cont_to_sd}
  Assume that $u_0$ is analytic and $f$ is uniformly analytic on a fixed neighborhood $\widetilde{\Omega} \supset \overline{\Omega}$
 (in the sense of Definition~\ref{def:uniformly_analytic}).
  Let $\HHy$ be given by Definition~\ref{def:geo-mesh-y}. 
  Fix $t_0>0$, $\delta >0$, and set $\zhf:=t_0^{-1/s} L^{1/s}$, where $L$ is the number of layers used for constructing the geometric 
mesh $\HHy$.
  Let the space $\HHx$ resolve the scales  down to \eqref{eq:scales_that_need_resolving},  
  and let Assumption~\ref{ass:approx_of_unitial_condtiion} hold for the initial condition.
  Then the following estimate holds:
  \begin{multline*}
    \int_{0}^{t}{\norm{ u(\tau) - u_h(\tau)}^2_{L^2(\Omega)} \,d\tau} \\
    \lesssim \max(1,t^2) \left(t_0^{\min(\frac{1}{2s}-\delta,1)}  + \abs{\log(t_0)}^2 \max(\log(t/t_0),0)\left[e^{-b_1 \mathcal{N}_{\Omega}^\mu}
        + e^{-b_2  \sqrt{\mathcal{N}_{\Omega}}} \right]\right).     
  \end{multline*}
\end{theorem}
\begin{proof}
  We just collect all the previous results, most notably Proposition~\ref{prop:error_up_to_ritz} and Corollary~\ref{cor:best_approx_inhomogeneous_rhs}.
  Since we only need the best approximation estimate on $\UU$ and  $\dot{\UU}$, we keep the dependence on the time $t$ explicit.
  The error due to the different initial conditions is exponentially small by assumption.
\end{proof}

We can also obtain estimates in the energy norm or pointwise in time: 
\begin{theorem}
  \label{thm:error_est_cont_to_sd_energy}
  Assume that $u_0$ is analytic, $f$ and $\dot{f}$ are uniformly analytic  on a neighborhood $\widetilde{\Omega} \supset \overline{\Omega}$,
  and that $u_{h,0} \in \mathbb{V}^{\mathcal{X}}_{h,\beta}$ is as in Assumption~\ref{ass:approx_of_unitial_condtiion}.

  Let $L$ denote the number of layers used for $\HHy$, set $t_0:=e^{-L}$, and  $\zhf:=t_0^{-1/s} L^{1/s}$
  and assume that the  space $\HHx$ resolves the scales down  to~\eqref{eq:scales_that_need_resolving}.
  
  Set $M:=\min(L,\operatorname{dim}(\HHx)^\mu)$ with $\mu>0$ from Assumptions~\ref{ass:approx_of_unitial_condtiion} and~\ref{ass:approx_of_HHx}.

  Then there exists a constant $b$, independent of $L$, $p$ and the specific choice of $\HHx$,
  i.e. depending only on the constants from Assumptions~\ref{ass:approx_of_unitial_condtiion} and~\ref{ass:approx_of_HHx} such that the following estimate holds:
  \begin{align*}
    \norm{u(t) - u_h(t)}_{L^2(\Omega)}^2+\int_{0}^{t}{\norm{ u(\tau) - u_h(\tau)}^2_{\widetilde{H}^{s}(\Omega)} \,d\tau} 
    &\lesssim \max(1,t^2 \log(t)) e^{-b M}.
  \end{align*}
\end{theorem}
\begin{proof}    
  Without loss of generality, we may assume $\beta s < 1/2$.
  Fix $t_1>0$ to be chosen later.  We consider two regimes, $t \in (0,t_1)$ and $t\geq t_1$.   For 
  $t \leq t_1$, we use the stability estimates of Lemma~\ref{lemma:apriori}~(\ref{it:sg_apriori_stability}),
  together with the insight that $u_0 \in \widetilde{H}^{\beta s}(\Omega)$ for $\beta s < 1/2$ 
  which was already used in Lemma~\ref{lemma:HHx_approx_small_t0}.
  
  We start with the energy norm estimate and use Lemma~\ref{lemma:apriori} to get:
  \begin{align*}
    \int_{0}^{t}{\!\norm{u(\tau)\!-\!u_h(\tau)}_{\widetilde{H}^s(\Omega)}^2 \! d\tau}
    &\lesssim \int_{0}^{t}{\norm{u(\tau)}_{\widetilde{H}^s(\Omega)}^2 + \norm{u_h(\tau)}_{\widetilde{H}^s(\Omega)}^2\,d\tau}  \\
    &\!\!\!\lesssim \int_{0}^{t}{\tau^{-1 + \beta} \big( \norm{u_0}_{\widetilde{H}^{\beta s}(\Omega)}^2 + \norm{u_{h,0}}_{\mathbb{V}^{\mathcal{X}}_{h,\beta}}^2\big)\,d\tau} 
    \lesssim t_1^{\beta} \big( \norm{u_0}_{\widetilde{H}^{\beta s}(\Omega)}^2 + \norm{u_{h,0}}_{\mathbb{V}^{\mathcal{X}}_{h,\beta}}^2\big).
  \end{align*}
  For the pointwise estimate, we write
  $u(t)=u_0+\int_{0}^{t}{\dot{u}(\tau) d\tau}$ and  $u_h(t)=u_{h,0}+\int_{0}^{t}{\dot{u}(\tau) d\tau}$ to get: 
  \begin{align*}
    \norm{u(\tau)-u_h(\tau)}_{L^2(\Omega)}
    &\lesssim \norm{u_0 - u_{h,0}}_{L^2(\Omega)} +  \int_{0}^{t}{\norm{\dot{u}(\tau)}_{L^2(\Omega)} + \norm{\dot{u}_h(\tau)}_{L^2(\Omega)}\,d\tau} \\    
    &\lesssim \norm{u_0 - u_{h,0}}_{L^2(\Omega)} +
      t_1^{\beta/2} \big( \norm{u_0}_{\widetilde{H}^{\beta s}(\Omega)} + \norm{u_{h,0}}_{\mathbb{V}^{\mathcal{X}}_{h,\beta}}\big).
  \end{align*}

For larger times $t>t_1$, we can establish the following bound by using \eqref{eq:error_up_to_ritz_pointwise_and_hone} and
plugging in the results on the best approximation from Corollary~\ref{cor:best_approx_inhomogeneous_rhs}.
  \begin{align*}
    t \norm{ u(t) - u_h(t)}^2_{L^2(\Omega)} \!+\! \int_{t_1}^{t}{\tau \norm{ u(\tau) - u_h(\tau)}^2_{\widetilde{H}^{s}(\Omega)} d\tau} 
    \lesssim \max(1,t^2\log(t))  e^{-b' M} \!\! + \!  t \norm{u_0 - u_{h,0}}_{L^2(\Omega)}^2.
  \end{align*}  
  Or, since $\tau > t_1$:
  \begin{align*}
    \norm{ u(t) \! - \! u_h(t)}^2_{L^2(\Omega)} \! + \!\int_{t_1}^{t}{\!\norm{ u(\tau) \! - \! u_h(\tau)}^2_{\widetilde{H}^{s}(\Omega)} d\tau} 
    \lesssim t_1^{-1} \max(1,t^2\log(t))  \big( e^{-b' M} \! + \!\norm{u_0 - u_{h,0}}_{L^2(\Omega)}^2 \! \big).
  \end{align*}
  Setting $t_1 \sim e^{-\frac{b'}{2} M}$ we get the stated exponential convergence with rate $b:=-b'\beta/2$
  after  using Assumption~\ref{ass:approx_of_unitial_condtiion} to estimate the error due to approximating the initial condition.
\end{proof}

\subsection{Example of a space $\HHx$: $hp$-FEM in 1D and 2D}
\label{sect:one_dimensional}
In this section, we give an exemplary construction for $\HHx$ given a simpler model problem in one or two space dimension using $hp$-Finite Elements
meeting our requirements.
In other words, $\HHx$ satisfies Assumptions~\ref{ass:approx_of_unitial_condtiion}~and~\ref{ass:approx_of_HHx}.

\begin{assumption}
  The domain $\Omega \subseteq \R^d$ for $d=1,2$ has analytic boundary.
  Also,  the coefficient functions $A$ and $c$ are analytic on a neighborhood $\widetilde{\Omega} \supset \overline{\Omega}$.
\end{assumption}

In 1D, we have already introduced $hp$-FEM spaces. For analytic 2D geometries they are given in the following 
Definition~\ref{def:reference-mesh}. We follow~\cite{MS98} and \cite{tensor_fem}; see also \cite[Definition~{2.4.1}]{melenk_book}.
We first introduce the (shape regular) reference mesh.
\begin{definition}[reference mesh]
\label{def:reference-mesh}
  Denote by $\widehat{S}:=(0,1)^2$ the reference square, and let $\TT_{\Omega}:=\big\{K_i\big\}_{i=0}^{\abs{\TT_{\Omega}}}$  be a mesh of curved quadrilaterals
  with bijective element maps $F_{K}: \overline{\widehat{S}} \to \overline{K}$ satisfying
  \begin{enumerate}[(M1)]
  \item The elements $K_i$ partition $\Omega$, i.e., $\bigcup_{K_i \in \TT} \overline{K} = \overline{\Omega}$;
  \item for $i\neq j$, $\overline{K}_i \cap \overline{K}_j$ is either empty, a vertex or an entire edge;
  \item the element maps $F_{K}:\widehat{S} \to K$ are analytic diffeomorphisms;
  \item the common edge of two neighboring elements $K_i$, $K_j$ has the same parametrization from both sides,
    i.e., if $\gamma_{ij}$ is the common edge with endpoints $P_1$, $P_2$, then for $P \in \gamma_{i,j}$ we have
    $$
    \operatorname{dist}(F_{K_i}^{-1}P,F_{K_i}^{-1} P_{\ell}))=\operatorname{dist}(F_{K_j}^{-1}P,F_{K_j}^{-1} P_{\ell}) \qquad \text{for } \ell=1,2.
    $$
  \end{enumerate}  
\end{definition}

\begin{definition}[anisotropic geometric mesh]
  \label{def:anisotropic_geometric_mesh}
  Given a reference mesh $\TT_{\Omega}$. Let $K_{i}$, $i=0,\dots, n < \abs{\TT_{\Omega}}$ be the elements at the boundary,
  and assume that the left edge $e:=\{0\}\times (0,1)$ is mapped to $\partial \Omega$, i.e., $F_{K_i}(e) \subseteq \partial \Omega$
  and $F_{K_i}(\partial S \setminus \overline{e}) = \emptyset$ for  $i=0,\dots, n$.
  Assume that the remaining elements satisfy $\overline{K_i} \cap \partial \Omega = \emptyset$, $i=n+1,\dots \abs{\TT_{\Omega}}$.

  For $L \in \N$ and a mesh grading factor $\sigma \in (0,1)$, we subdivide the reference square
  \begin{align*}
    \widehat{S}^0:=(0,\sigma^{L}) \times (0,1),\; \quad \widehat{S}^{\ell}:=(\sigma^{\ell},\sigma^{\ell-1})\times (0,1),\; \ell=1,\dots. L.
  \end{align*}

  The anisotropic geometric mesh $\TT_{\Omega}^{L}$ is then given by the push-forwards of the refinements plus the unrefined interior elements:
  $$
  \TT_{\Omega}^{L}:=\Big\{ F_{K_i}(\widehat{S}^{\ell}), \; \ell=0,\dots,L,\;\; i=0,\dots,n\big\} 
   \cup \bigcup_{i=n+1}^{\abs{\TT_{\Omega}}}{\{ K_{i}\} }.
  $$  
\end{definition}

\begin{definition}
  In one dimension, for $\Omega=(a,b)$, the anisotropic mesh is defined as in Section~\ref{sect:disc_of_y} by  $\TT_{\Omega}^{L}:=\TT_{(\mathbf{a}, \mathbf{b})}^{L}$.
  The reference mesh is given by the single element $\TT_{\Omega}:=\big\{(a,b)\big\}$.
\end{definition}

We are now able to define the space $\HHx$ using these meshes.
\begin{definition}[$\HHx$ via $hp$-FEM]
  Let $\TT_{\Omega}^L$ be an anisotropic geometric mesh refined towards $\partial \Omega$ and fix $p\in \N$.
  We write
    $
    \QQ^p:=\operatorname{span}_{0 \leq i_1, \,\dots,i_d\leq p }\big\{x_1^{i_1} \, \dots \,x_{d}^{i_d}\big\}
    $ for the space of tensor product polynomials and set
  \begin{align}
    \label{eq:def_hhx_for_fem}
    \HHx:=S^{p,1}_0(\TT^{L}_{\Omega}):=\Big\{ u \in H_0^{1}(\Omega): \; u\circ F_K \,\in\, \QQ^p \quad \forall K \in \TT_{\Omega}^{L} \big\}.
  \end{align}
\end{definition}

This choice of approximation space will prove suitable to satisfy Assumptions~\ref{ass:approx_of_unitial_condtiion} and~\ref{ass:approx_of_HHx}.
We start with the fact that we can resolve certain scales: 
\begin{theorem}
  \label{thm:hp_fem_resolves_scales}
  Let $\TT^{L}_{\Omega}$ be an anisotropic mesh on $\Omega$ that is geometrically refined towards $\partial \Omega$ with  grading factor $\sigma \in (0,1)$
  and $L$ layers.
  
  Then $\HHx$ defined in~\eqref{eq:def_hhx_for_fem} resolves the scales  down  to $\sigma^{L}$, i.e.,
  there exist constants $C, b > 0$, such that
  for  $z \in \mathscr{S}$ with ${\abs{z}}^{-1/2}>\sigma^{L}$ 
  and  every $f$ which is analytic on a neighborhood $\widetilde{\Omega}$ of $\Omega$, the solution $u_z$ to
  $(\LL - z)  u = z f$  can be approximated by $v_h \in \HHx$ satisfying
  \begin{align*}
    \abs{z}^{-1} \norm{\nabla u- \nabla v_h}_{L^2(\Omega)}^2 +\norm{u - v_h}^2_{L^2(\Omega)}
    &\leq C e^{- b' L} \sim e^{-b \mathcal{N}_{\Omega}^{\frac{1}{d+1}}}.
  \end{align*}
  The constant $b$ depends only on $\sigma$ and $\widetilde{\Omega}$. The constant $C$ also depends
  on the constants of analyticity of $f$.
\end{theorem}
\begin{proof}
  We consider two cases. For $\abs{z}<2 z_0$, the problems are not actually singularly perturbed and standard
  results for $hp$-FEM can be applied.
  We thus only focus on the case $\abs{z} > 2 z_0$.

  From the definition, we get 
  $$
  (\LL - z)  u= -\fdiv{(A \nabla u)} + (c-z) u.
  $$
  Defining $\varepsilon:=\abs{z}^{-1/2}$ and $\zeta:=\varepsilon^2 (c-z)$, the problem can be rewritten as
  $$
  - \varepsilon^2 \fdiv{(A\nabla u)} + \zeta u = \varepsilon^2 z f=:\widetilde{f}.
  $$
  We make the following observations:
  \begin{enumerate}[(i)]
  \item $\abs{\varepsilon^2 z} \sim 1$ and therefore $\|\widetilde{f}\| \sim \|f\|$ for any norm,
  \item since we assumed $\abs{z} > 2 z_0$, by slightly decreasing the opening angle, we may ignore the shift in the definition of $\mathcal{S}$
    and assume that $\abs{\pi - \operatorname{Arg}(\zeta)} \geq \delta > 0$; see Figure~\ref{fig:proof_of_sing_decomp}.  
  \end{enumerate}
  This means that we can apply the results from Appendix~\ref{appendix:oned_singular_perturbations}, most notably Theorem~\ref{lemma:approx_for_sing_pert}.  
\end{proof}

\begin{figure}
\begin{center}
  \includeTikzOrEps{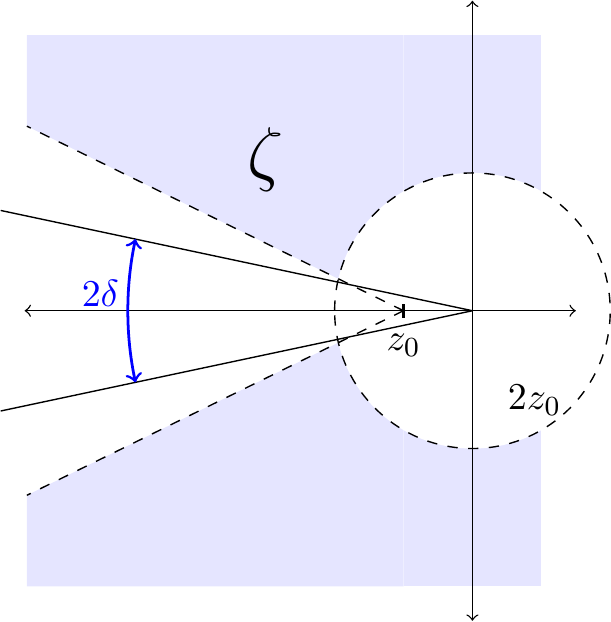}
\end{center}
  \caption{The geometric situation in the proof of  Theorem~\ref{thm:hp_fem_resolves_scales} (for $c=0$).} 
  \label{fig:proof_of_sing_decomp}
\end{figure}

The $hp$-FEM spaces can also approximate the initial conditions at an exponential rate. 
But more importantly, they can do so in a way that is stable with respect to the
non-standard $\mathbb{V}^{\mathcal{X}}_{h,\beta}$ norm. We start with a simple lemma.
    
\begin{lemma}
  \label{lemma:lifting_has_minimal_energy}
  Given $u \in \HHx$, for any function $\VV \in \HHxy$ with $\trace{\VV}=u$, we can estimate
  $$\norm{u}_{\HHx} \leq \norm{\VV}_{\HH}.$$
  In other words, $\VV=\lifting_h  u$ has ``minimal energy''.
\end{lemma}
\begin{proof}
  We compute for $\WW \in \HHxy$ with $\trace{\WW}=0$:
  \begin{align*}
    \AA(\lifting_h u - \WW,\lifting_h u - \WW) 
      &= \AA(\lifting_h u,\lifting_h u) - 2 \AA(\lifting_h u,\WW) + \AA(\WW,\WW) \\
      &= \AA(\lifting_h u,\lifting_h u)  + \AA(\WW,\WW)
      \;\geq\;  \AA(\lifting_h u,\lifting_h u),
  \end{align*}
  where we used $\AA(\lifting_h u, \WW)=0$ for $\WW \in \HHxy$ with $\trace{\WW}=0$ by the definition of the lifting. Setting 
  $\WW:=\lifting_h u - \VV$ then shows the estimate
  $\norm{u}_{\HHx} \leq \norm{\VV}_{\HH}$.
\end{proof}

Working with anisotropic meshes for the discrete liftings imposes additional difficulties. Instead, we split the lifting process into
two steps, first we lift using the shape-regular reference triangulation (but ignoring boundary conditions), and then we
use a cutoff procedure to correct the boundary conditions on the anisotropic geometric mesh. This cutoff operator is constructed in the following Lemma.
\begin{lemma}
  \label{lemma:cutoff_operator}
  Let $\TT^L_{\Omega}$ denote an anisotropic geometric mesh with reference mesh $\TT_{\Omega}$.
  Given $\ell \in \N_0$, $\ell\leq L$, there exists a bounded linear operator $\mathscr{C}_{\ell}: \SS^{p,1}(\TT_{\Omega}) \to \SS^{p,1}_0(\TT^L_{\Omega})$ such that
  \begin{align}
    \label{eq:stability_cutoff_operator}
    \norm{\mathscr{C}_{\ell} v}_{H^1(\Omega)} &\lesssim \sigma^{-\ell/2} \norm{v}_{L^2(\Omega)} + \norm{\nabla v}_{H^1(\Omega)},
  \end{align}
  and for all $0\leq \beta < 1/2$
  $$
  \norm{v - \mathscr{C}_{\ell} v}_{L^2(\Omega)} \lesssim \sigma^{{\beta \ell}} \norm{v}_{H^{\beta}(\Omega)}.
  $$
  
\end{lemma}
\begin{proof}
  We fix a layer of thickness $\sigma^\ell$ around the boundary and  
  pick a piecewise affine function $\chi$ such that $\chi = 1$ on all elements outside of this layer. This can be easily done working on the reference patches.  
  This leads to a function which only has non-vanishing gradient on this layer, and there satisfies the estimate
  $$
  \norm{\nabla \chi}_{L^{\infty}(\Omega)}\lesssim \sigma^{-\ell}.
  $$
  Working on the reference element, we note that the Gauss Lobatto in one direction satisfies the following stability estimate, also on anisotropic elements $K$:
  \begin{align*}
    \norm{i_{p} \otimes \id \, u}_{L^2(K)} &\lesssim  \bigg(\frac{q}{p}\bigg)^2 \norm{u}_{L^{2}(K)} \quad  \text{ and }\quad
    \norm{i_{p} \otimes \id \,u}_{H^1(K)} \lesssim  \frac{q}{p} \norm{u}_{H^1(K)}  \qquad  \forall u \in \QQ_{q}(K).
  \end{align*}
  (see \cite[Rem.~{13.5} and (13.27)]{bernardi-maday97} for the 1D case.
  2D then follows by a tensor product argument, see also~\cite[Eqn. (15)]{HS03}).
  Combining such element-wise Gauss-Lobatto interpolants on each element, we get a global stable interpolation operator $\Pi_p$.
  
  We then define $\mathscr{C}_{\ell} v:=\Pi_{p}(\chi v)$. We note that since $\chi \,v$ is a piecewise polynomial of degree $p+1$, we get the stability estimates:
  \begin{align*}
    \norm{\mathscr{C}_{\ell} v}_{L^2(\Omega)} & \lesssim \norm{\chi \, v}_{L^2(\Omega)} \qquad \text{and}\qquad \norm{\mathscr{C}_{\ell} v}_{H^{1}(\Omega)}  \lesssim \norm{\chi \, v}_{H^{1}(\Omega)}.
  \end{align*}
  From the estimate on $\nabla \chi$ we then immediately get \eqref{eq:stability_cutoff_operator}.
  
  The approximation estimate follows from the fact that $\Pi_p$ reproduces $v$ and is $L^2$ stable and the following
  estimate taken from~\cite[Lemma 2.1]{lmwz10}. Since $1-\chi$-vanishes outside of the strip size $\sigma^{\ell}$:
  \begin{align*}
  \norm{(1-\chi) u}_{L^2(\Omega)}& \lesssim \sigma^{\beta \ell} \norm{u}_{H^{\beta}(\Omega)}. \qedhere
  \end{align*}
\end{proof}

\begin{lemma}
  \label{lemma:inital_condition_in_right_space}
  Assume that the triangulation $\TT^{L}_{(\mathbf{0},\YY)}$ used for the discretization in $y$ satisfies
  $\sigma^L \lesssim  p_x^{-2}$, where $p_x$ denotes the (maximal) polynomial degree used for $\HHx$.

  Let $u_0$ be analytic in a neighborhood $\widetilde{\Omega} \supset \overline{\Omega}$, and let $0\leq \beta <1/2$
  and $\varepsilon >0$.
  Then there exists a function $u_{h,0} \in \HHx$ :
  \begin{align}
    \label{eq:inital_condition_in_right_space}
    \norm{u_{h,0}}_{\mathbb{V}^{\mathcal{X}}_{h,\beta}} \lesssim \norm{u_0}_{H^{\max(s,\beta+\varepsilon)}(\Omega)} 
    \qquad \text{ and } \qquad \norm{u_{h,0} - u_0}_{L^2(\Omega)} \lesssim e^{-b' p_x }.
  \end{align}

  In other words, if the number of refinement layers $L \sim p_x$, then
  $\HHx$ satisfies Assumption~\ref{ass:approx_of_unitial_condtiion}   with $\mu:=1/(d+1)$.
\end{lemma}
\begin{proof}
Since $u_0$ is analytic, we do not need to approximate any boundary layers or singularities.
What we do need to take care of is the fact that our FEM space has homogeneous boundary conditions, while $u_0$ does not.

We will construct the lifting in two steps: First, we approximate and lift in a space without boundary conditions and
then we will perform a cutoff procedure.

For $\varepsilon \in (0, 1/2-\beta)$,
Let $\widehat{u}_{0} \in \SS^{p,1}(\TT_{\Omega})$  be the $H^{\max(s,\beta+\varepsilon)}(\Omega)$-best approximation of $u_0$. (Note that we work on the shape-regular grid $\TT_{\Omega}$ and do not impose boundary conditions.)
By standard results, we have $\norm{u_0-\widehat{u}_0}_{L^2(\Omega)}\lesssim e^{-b \N_{\Omega}^\mu}$ and by Lemma~\ref{lemma:lifting_of_polynomials}
we can lift this function to $\widehat{\UU}_0 \in \HHnd$ such that $\widehat{\UU}_0(\cdot,y) \in \SS^{p,1}(\TT_{\Omega})$ and $\widehat{\UU}_{0}(x,\cdot) \in \SS^{1,1}(\TT_{\YY})$
for all $y \in (0,\YY)$ and $x \in \Omega$.

Using the cutoff operator from Lemma~\ref{lemma:cutoff_operator} we then define $u_{h,0}:=\mathscr{C}_{L} (\widehat{u}_0)$ and
the piecewise constant function $\VV(t): \R_+ \to \HHxy$ 
\begin{align*}  
  \VV(t,y):=\begin{cases}
    \mathscr{C}_{L} \big( \widehat{\UU}_0(\cdot,y) \big) & t \in (0, \sigma^L), \\
    \mathscr{C}_{\ell} \big( \widehat{\UU}_0(\cdot,y) \big) & t \in (\sigma^{\ell+1},\sigma^{\ell}), \; \ell=0,\dots,L-1, \\
    0 & t>1.
  \end{cases}
\end{align*}
Since $\mathscr{C}_{\ell}$ is $L^2$ and $H^1$ stable and is applied on each $y$-slice, we get stability in the $\HHnd$-norm,
i.e. $\norm{\VV(t)}_{\HHnd}\lesssim \sigma^{-\ell/2} \big\|{\widehat{\UU}_0}\big\|_{\HHnd}$.
We also get the approximation of the trace at $y=0$ for $t \in (\sigma^{\ell+1},\sigma^{\ell})$:
\begin{align}
  \label{eq:discrete_lifting:approximation_property_of_traces}
  \norm{u_{h,0} - \trace{\VV}(t)}_{L^2(\Omega)}
  &\lesssim \norm{\mathscr{C}_L \widehat{u}_0 - \mathscr{C}_{\ell} \widehat{u}_0 }_{L^2(\Omega)}
   \lesssim \sigma^{(\beta+\varepsilon) \ell}  \norm{\widehat{u}_0}_{H^{\beta+\varepsilon}(\Omega)}.
\end{align}
We then need to estimate the $\mathbb{V}^{\mathcal{X}}_{h,\beta}$-norm. By construction of the cutoff function,  we get that $\VV(t) \in \HHx$ and calculate:
\begin{align*}
  \norm{u_{h,0}}_{\mathbb{V}^{\mathcal{X}}_{h,\beta}}^2 
  &\lesssim \int_{0}^{\infty}{ t^{-2\beta -1} \big( \norm{u_{h,0} - \trace{\VV(t)}}_{L^2(\Omega)}^2 + t^2 \norm{\VV(t)}^2_{\HH}\big)\,dt} \\
    &= \int_{\sigma^{L}}^{\infty}{ t^{-2\beta -1 } \norm{u_{h,0} - \trace{\VV(t)}}^2_{L^2(\Omega)}\,dt} + \int_{0}^{1}{ t^{-2\beta +1 } \norm{\VV(t)}_{\HH}^2\,dt}
\end{align*}
where we used that we can replace the specific lifting $\lifting_h$ with the function $\VV$ as it has the ``minimum energy'' property via Lemma~\ref{lemma:lifting_has_minimal_energy}.

From the stability estimates on each segment $(\sigma^{\ell+1},\sigma^{\ell})$, we get:
\begin{align*}
  \int_{0}^{1}{ t^{-2\beta +1 } \norm{\VV(t)}_{\HH}^2\,dt}
  &\lesssim
    \bigg(\sigma^{-L}\int_{0}^{\sigma^{L}}{ t^{-2\beta +1}\, dt}  
    + \sum_{\ell=0}^{L-1} { \sigma^{-\ell}  \int_{\sigma^{\ell+1}}^{\sigma^{\ell}}{  t^{-2\beta +1} \,dt}} \bigg)\norm{\widehat{\UU}_0}_{\HH}^2\\
  &\lesssim  \bigg(\sigma^{-L+ (-2\beta +2)L }+  \sum_{\ell=0}^{L-1} {\sigma^{-\ell +(-2\beta +2)\ell }} \bigg)  \norm{\widehat{\UU}_0}_{\HH}^2 \\
    &\lesssim \norm{\widehat{u}_0}_{H^{s}(\Omega)}^2
    \lesssim \norm{{u}_0}_{H^{\max(s,\beta+\varepsilon)}(\Omega)}^2.
\end{align*}
Where  we used $\beta < 1/2$, a geometric series, and the stability of the lifting of $\widehat{u}_0$ and the best approximation.
From the approximation property~\eqref{eq:discrete_lifting:approximation_property_of_traces} we get:
\begin{align*}
  \int_{\sigma^L}^{1}{ t^{-2\beta -1 } \norm{u_{h,0} - \trace{\VV(t)}}_{L^2(\Omega)}^2\,dt}
  &\lesssim    
    \bigg(\sum_{\ell=0}^{L-1} { \sigma^{2\ell (\beta+\varepsilon)}  \int_{\sigma^{\ell+1}}^{\sigma^{\ell}}{  t^{-2\beta -1} \,dt}} \bigg) \norm{\widehat{u}_0}_{H^{\beta+\varepsilon}(\Omega)}^2\\
  &\lesssim \bigg(\sum_{\ell=0}^{L-1} { \sigma^{2\ell (\beta+\varepsilon)} \sigma^{-2\beta\ell} }\bigg)  \norm{\widehat{u}_0}_{H^{\beta+\varepsilon}(\Omega)}^2
    \lesssim \norm{\widehat{u}_0}_{H^{\beta+\varepsilon}(\Omega)}^2.
\end{align*}
The estimate $\int_{1}^{\infty}{t^{-2\beta-1}\norm{u_{h,0}}^2_{L^2(\Omega)}\,dt}\lesssim \norm{u_{h,0}}_{L^2(\Omega)}^2\lesssim \norm{\widehat{u}_{0}}_{L^2(\Omega)}^2$ is trivial.

The approximation estimate from~\eqref{eq:inital_condition_in_right_space} then follows from the approximation property of $\mathscr{C}_{L}( \widehat{u}_{0})$ from Lemma~\ref{lemma:cutoff_operator} and the best approximation property of $\widehat{u}_{0}$.

\end{proof}

\begin{remark}
  For constructing the lifting, Lemma~\ref{lemma:inital_condition_in_right_space} relies on the as of yet unpublished work~\cite{our_interpolation_space_paper}.
  In the simpler, one dimensional case, the space $\SS^{p,1}(\TT_\Omega)$ coincides with the space $\QQ^p$ of global polynomials since the reference mesh only consists of a single  element.
  This allows us to replace \cite{our_interpolation_space_paper} with results from \cite{bernardi:hal-00153795} in this case.
\eremk
\end{remark}

We can now give a more constructive characterization of how the triangulation of $\Omega$ must be chosen
when working in 1D or 2D  to get exponential convergence of the semidiscretization.
\begin{corollary}
  \label{cor:constructive_error_est_cont_to_sd}
  Let $\Omega\subset \R^d$, $d=1,2$ have analytic boundary.
  Assume that $u_0$ is analytic and $f$ is uniformly analytic in a neighborhood $\widetilde{\Omega} \supset \overline{\Omega}$.    
  For $M \in \N$ and $\sigma \in (0,1)$, use an anisotropic geometric mesh with $M$ layers to discretize in $x$, i.e., $\HHx:=\SS^{p,1}_{0}(\TT^{M}_{\Omega})$.
  For discretizing in $y$, use $L$ layers and a degree vector $\mathbf{r}$ with linear slope $\mathfrak{s}$,
  i.e., $\HHy:=\SS^{\mathbf{r},1}(\TT^{L}_{(\mathbf 0,\mathcal{Y})})$.
  Assume that $\sigma^M \leq c_\YY \YY (\mathfrak{s} L)^{-2} \sigma^{3L/2}$, $\sigma^L < p^2$
  and $u_{h,0}$ is as in Assumption~\ref{ass:approx_of_unitial_condtiion}.

  Then there exist constants $b_1,b_2 >0$ independent of $L$, $M$, and $p$ such that the following estimate holds:
  \begin{align*}
    \int_{0}^{t}{\norm{ u(\tau) - u_h(\tau)}^2_{L^2(\Omega)} \,d\tau}
    &\lesssim \max(1,t^2\log(t)) \left(e^{-b_1 p} + e^{-b_2 L} \right).
  \end{align*}
  Most notably for $M \sim \frac{3}{2} L $ and $p \sim L$, we get exponential convergence:
  \begin{align*}
    \int_{0}^{t}{\norm{ u(\tau) - u_h(\tau)}^2_{L^2(\Omega)} \,d\tau}
    &\lesssim \max\big(1,t^2{\log(t)}\big) e^{-b' \dim(\HHxy)^\frac{1}{d+3}}.
  \end{align*}
\end{corollary}
\begin{proof}
  We choose $t_0:=\sigma^L$ and $\zhf=\sigma^{L/s} L^{1/s}$ in Theorem~\ref{thm:error_est_cont_to_sd}.
  Assumption~\ref{ass:approx_of_unitial_condtiion} is met via Lemma~\ref{lemma:inital_condition_in_right_space}.
  The assumptions on $\HHx$ also imply that the necessary scales are resolved, and we get:
  \begin{multline*}
    \int_{0}^{t}{\norm{ u(\tau) - u_h(\tau)}^2_{L^2(\Omega) }\,d\tau} \\
    \lesssim \max(1,t^2)\sigma^{\left(\frac{1}{2s}-\delta\right) M}  
    + \max(1,t^2)\abs{\log(t_0)}^2\max\big(\log(t/t_0),0\big)\left[e^{-b_1 p} + e^{-b_2 L} + e^{-L}\right].     
  \end{multline*}
  The explicit estimate then follows from the fact that $\dim(\HHx) \sim L^{d+1}$ in this
  particular construction and $ \dim(\HHy) \sim L^2$. We absorb the logarithmic terms $\log(\sigma^L)\sim L$ into the exponential
  by slightly reducing the rate $b'$.
  The condition $\sigma^L \leq p^{-2}$ is easily verified for such meshes.
\end{proof}

For the pointwise and energy errors, the corresponding concrete version reads:  
\begin{corollary}
  \label{cor:error_est_cont_to_sd_energy_1d}
  Assume that $u_0$ is analytic and $f$, $\dot{f}$ are uniformly analytic in a neighborhood $\widetilde{\Omega} \supset \overline{\Omega}$,    
  and that the meshes and spaces are as in Corollary~\ref{cor:constructive_error_est_cont_to_sd}.
  Let $u_{h,0} \in \mathbb{V}^{\mathcal{X}}_{h,\beta}$ be as in Assumption~\ref{ass:approx_of_unitial_condtiion} for $\beta >0$.

  Then there exists a constant $b$, independent of $L$, $M$ and $p$ such that the following estimate holds:
  \begin{align*}
    \norm{u(t) - u_h(t)}_{L^2(\Omega)}^2+\int_{0}^{t}{\norm{ u(\tau) - u_h(\tau)}^2_{\widetilde{H}^{s}(\Omega)} \,d\tau} 
    &\lesssim \max(1,t^2 \log(t)) e^{-b L}.
  \end{align*}
  
  Or in terms of degrees of freedom, we get
  \begin{align*}
    \norm{u(t) - u_h(t)}_{L^2(\Omega)}^2 + \int_{0}^{t}{\norm{ u(\tau) - u_h(\tau)}^2_{\widetilde{H}^{s}(\Omega)} \,d\tau}
    &\lesssim \max\big(1,t^2\log(t)\big) e^{-b' \dim(\HHxy)^{\frac{1}{d+3}}} .
  \end{align*}  
\end{corollary}
\begin{proof}
  Follows from the fact that using the given parameters, the space $\HHx$ satisfies the assumptions of Theorem~\ref{thm:error_est_cont_to_sd_energy}.
  The estimate in terms of degrees of freedom follows easily.
\end{proof}

\begin{remark}
  In this section, we focused on the case of smooth geometries in 1D and 2D. We would like to point out that we do not see any structural
  obstacles towards generalizing to the case of curvilinear polygons or smooth 3d geometries. The main ingredient is the necessary generalization
  of Appendix~\ref{appendix:oned_singular_perturbations}.
\eremk
\end{remark}

\section{Discretization in $t$ -- the fully discrete scheme}
\label{sect:time_discretization}
In this section, we consider the discretization with respect to the time variable $t$. This can be done using 
mostly standard techniques. We focus on the case of using a discontinuous Galerkin type method.
When applied in its $hp$-version, it will allow us to get an exponentially convergent fully discrete scheme, and thus it nicely
complements our previous investigations. We follow the presentation in~\cite{schoetzau_schwab}.

Let $\TT_{(0,T)}:=\{(t_{j-1},t_{j}) \}_{j=1}^{M}$ be a partition of the time interval $[0,T]$ into subintervals with $0\leq t_j<t_{j+1}\leq T$.
We set $k_j:=t_{j}-t_{j-1}$ and define the one-sided limits
\begin{align*}
  u_j^+&:=\lim_{h \to 0, h> 0} u(t_j +h)  \qquad  \text{for $0\leq j\leq M-1$}, \\
  u_j^-&:=\lim_{h \to 0, h> 0} u(t_j - h) \qquad \text{for $1\leq j\leq M$}
\end{align*}
as well as the jump $[u]_j:=u_j^+-u_j^-$.
We define the DG-bilinear and linear forms:
\begin{align*}
  B(\UU,\VV)&:=\begin{multlined}[t][13cm]\sum_{j=1}^{M}{\int_{t_{j-1}}^{t_j}{ \ltwoprodX{\dot{\trace{\UU}(t)}}{\trace{\VV}(t)}} + d_s^{-1}\AA(\UU(t),\VV(t)) \,dt} \\
            + \sum_{j=2}^{M}{\ltwoprodX{[\trace{\UU}]_{j-1}}{\trace{\VV}_{j-1}^+} + \ltwoprodX{\trace \UU_0^+}{\trace{\VV}_0^+}}, \end{multlined}\\
  F(\VV)&:=\sum_{j=1}^{M}{\int_{t_{j-1}}^{t_j}{\ltwoprodX{f(t)}{\trace{\VV}(t)}  \,dt} + \ltwoprodX{u_0}{\trace{\VV}_0^+}}. \\
\end{align*}

Then the DG-approximation is given as the solution to the following problem:
\begin{problem}
  \label{prob:dg_formulation}
  Choose $\mathbf{r}_t \subseteq \N_0$ a polynomial degree distribution, and consider the space $\SS^{\mathbf{r}_t,0}(\TT_{(0,T)})$
  of discontinuous piecewise polynomials. Set $\HHxyt:=\SS^{\mathbf{r}_t,0}(\TT_{(0,T)}) \otimes \HHxy$. 
  Find $\UUdxyt \in \HHxyt$ such that
  \begin{align}
    \label{eq:dg_formulation}
    B(\UUdxyt, \VV_h) &= F(\VV_h) \qquad \qquad \forall \VV_h \in \HHxyt.
  \end{align}
\end{problem}
\begin{remark}
    Note that we used the initial condition $u_0$ instead of the discrete initial condition $u_{h,0}$. This is due to the fact
    that we need assumptions on $u_{h,0}$ which make it non-computable in practice. When we talk about
    ``equivalence to time discretization of the semidiscrete problem'' we always mean ``up to changing the initial condition'',
    which incurs an additional (but easily treatable) error term.
\eremk
  \end{remark}

\begin{lemma}
  \label{lemma:dg_formulation_using_lh}
  Problem~\ref{prob:dg_formulation} is equivalent to solving the ``standard'' DG-formulation for the semidiscrete semigroup~\eqref{eq:sd_semigroup}, i.e., 
  if we define
  \begin{align*}
    \widetilde{B}(U,V)&\begin{multlined}[t][13cm]:=\sum_{j=1}^{M}{\int_{t_{j-1}}^{t_j}{ \ltwoprodX{\dot{U}(t)}{V(t)}} + \ltwoprodX{\LL^s_h U(t)}{V(t)} \,dt} \\
                      + \sum_{j=2}^{M}{\ltwoprodX{[U]_{j-1}}{V_{j-1}^+}} + \ltwoprodX{U_0^+}{V_0^+}, \end{multlined}\\
    \widetilde{F}(V)&:=\sum_{j=1}^{M}{\int_{t_{j-1}}^{t_j}{\ltwoprodX{f(t)}{V(t)}  \,dt} + \ltwoprodX{u_0}{V_0^+}}. 
  \end{align*}
  Then $u_{h,k}:=\trace(\UUdxyt) \in \SS^{\mathbf{r},0}(\TT_{(0,T)})\otimes \HHx$ solves
  \begin{align}
    \label{eq:dg_formulation_using_lh}
    \widetilde{B}(u_{h,k},v_h)&=\widetilde{F}(v_h) \qquad \forall v_h \in \SS^{\mathbf{r},0}(\TT_{(0,T)}) \otimes \HHx.
  \end{align}
  On the other hand, we can recover the extended function by $\UUdxyt:=\lifting_h u_{h,k}$.
\end{lemma}
\begin{proof}
  We first show that $\lifting_h u_{h,k}$ solves Problem~\ref{prob:dg_formulation}. 

  Comparing the two formulations, the only interesting term is $\AA(\lifting_h u_{h,k},\VV)$.
  We note that we can write:
  \begin{align*}
    \AA(\lifting_h u_{h,k},\VV_h)&=\AA(\lifting_h u_{h,k},\VV_h-\lifting_h \trace \VV_h) +  \AA(\lifting_h u_{h,k},\lifting_h \trace \VV_h)\\
    &=\AA(\lifting_h u_{h,k},\lifting_h \trace \VV_h) = d_s\ltwoprodX{ \LL^s_h u_{h,k}}{\trace \VV_h},
  \end{align*}
  where  we used that  $\AA(\lifting_h u_{h,k},\WW_h)=0$ vanishes for functions with $\trace{\WW_h}=0$ by the definition of the lifting.
  Thus all the terms in the formulation directly correspond to each other.

  We now show the other direction. Let $\UUdxyt$  be a solution to Problem~\ref{prob:dg_formulation}.
  We pick a function $q$, such that $q(t)=0$ outside of a single interval $(t_{j-1},t_j)$ on which $q(t)$ is a polynomial.
  We then  test~\eqref{eq:dg_formulation} with functions of the form $\VV_h(t):=q(t) \VV_0$, where $\VV_0 \in \HHxy$ satisfies $\trace{\VV_0}=0$.
  This means that $\VV_h(t) \in  \HHxyt$ and we get, since all the terms involving $\trace{\VV_h}$ vanish:
  \begin{align*}
    \int_{t_{j-1}}^{t_j}{\AA(\UUdxyt(t),\VV_0) q(t) \,dt}&=0.
  \end{align*}
  Since $\UUdxyt(t)$ is a polynomial of degree $r_j$ in $t$, $\AA(\UUdxyt(t),\VV_h)$ also is such a polynomial.
  Since the integral vanishes when tested with all similar polynomials, we get that $\AA(\UU(t),\VV_h)=0$ for
  all $t \in (t_{j-1},t_j)$ and all admissible $\VV_0$.
  This means we can write $\UUdxyt=\lifting_h \trace{\UUdxyt}$ and we can proceed as before to match all the terms
  in the formulation to their counterpart.
\end{proof}

\begin{theorem}[$h$-version]
  \label{thm:dg_fd_to_sd_1}
  Let $u_h$ denote the semidiscrete solution to \eqref{eq:sd_semigroup}.
  Suppose that Assumption~\ref{ass:approx_of_unitial_condtiion} is fulfilled with $\beta >0$.
  Let $\mathbf{r}_t=r\equiv \text{const}$ be a fixed parameter.
  Choose $\TT_{(0,T)}$ as a graded mesh with the grading function $h(t):=t^{\beta\left(2 r +3\right)}$.
  Let $N:=\dim(\SS^{\mathbf{r}_t,0}(\TT_{(0,T)}))$.

  Assume $u_0$ is analytic in $\overline{\Omega}$ and that the right-hand side $f$ satisfies 
  \begin{align*}
    \norm{f^{(\ell)}(t)}_{L^2(\Omega)}\leq C d^{\ell} \Gamma(\ell +1)   \qquad \forall t\in [0,T], \ell \in \N_0,
  \end{align*}
  with constants $C$ and $d$ independent of $\ell$ and $t$.

  Then the following error estimate holds:
  \begin{align*}
    \sqrt{\int_{0}^{T}{\norm{u_h(t) - u_{h,k}(t)}_{\widetilde{H}^{s}(\Omega)}^2}}
    &\lesssim N^{-(r+1)} +  e^{-b \mathcal{N}_{\Omega}^\mu}.
  \end{align*}
  The implied constant depends on $\Omega$, $u_0$, $f$, $r$, the terminal time $T$, and the constants from Assumption~\ref{ass:approx_of_unitial_condtiion}.
\end{theorem}
\begin{proof}
  We note that $u_{h,0} \in \mathcal{V}_{h,\beta}^{\mathcal{X}}$ by Assumption
  and also that the solution to DG-formulation depends continuously on the initial condition.
  This last statement can be easily seen from the coercivity of $\widetilde{B}$ as shown in \cite[Lemma 2.7]{schoetzau_schwab}.
  Thus, up to an additional error term $C(T)\norm{\Pi_{L^2} u_0 - u_{h,0}}_{L^2(\Omega)}^2$ we may use $u_{h,0}$ as our initial condition.
  (This error term is exponentially small by Assumption~\ref{ass:approx_of_unitial_condtiion}).

  We want to apply the results from~\cite{schoetzau_schwab} and translate our setting into their requirements.
  They require separable Hilbert spaces $X \subseteq H$ with continuous, dense and compact embedding and a bilinear form $a(\cdot,\cdot): X \times X \to \C$,
  such that
  \begin{align*}
    \abs{a(u,v)}&\lesssim \norm{u}_{X}\norm{v}_{X},  \qquad \Re(a(u,u))\geq c \norm{u}_{X}^2, \qquad \text{and}\quad  a(u,v)=\overline{a(v,u)}                  
  \end{align*}
  for all $u,v \in X$.
  We set $H:=\left(\HHx,\norm{\cdot}_{L^2(\Omega)}\right)$, $X:=\left(\HHx,\norm{\cdot}_{\HHx}\right)$ 
  and $a(u,v):=\ltwoprodX{\LL^s_h u}{v}$ (extending the real valued bilinear form to a complex one in the canonical way).
  By Lemma~\ref{lemma_llh_is_elliptic} this bilinear form satisfies the boundedness and ellipticity conditions. The symmetry follows from
  the definition and the symmetry of $\AA(\cdot,\cdot)$.
  
  The stated result then is a consequence  of~\cite[Theorem 5.10]{schoetzau_schwab}.
  The main ingredient is the fact that the initial condition is in the interpolation space
  $\mathcal{V}^{\mathcal{X}}_{h,\beta}$ by Assumption~\ref{ass:approx_of_unitial_condtiion}. 
  Note that \cite[Theorem 5.10]{schoetzau_schwab} gives an estimate in the $\HHx$-norm.
  In order to get to the more natural $\widetilde{H}^{s}(\Omega)$-norm, we use Lemma~\ref{lemma_llh_is_elliptic}.
\end{proof}

\begin{remark}
    For $r:=1$, the scheme in Theorem~\ref{thm:dg_fd_to_sd_1} is equivalent to the more common implicit Euler discretization,
    except that the right hand side is slightly modified. See~\cite[Page 205]{thomee_book} for details.    
\eremk
  \end{remark}

\begin{theorem}[$hp$-version]
\label{thm:dg_fd_to_sd_2}
  Let $u_h$ denote the semidiscrete solution to \eqref{eq:sd_semigroup}.
  Consider $\TT_{(0,T)}:=\TT^M_{(\mathbf{0},t_1)} \cup  \TT_{(t_1,T)}$ to be a mesh on $(0,T)$ that is geometrically refined towards $0$ and
  has constant size for larger times $(t_1,T)$.
  We choose $\mathbf{r}_t$ such that it is linearly increasing on the geometrically refined part
  and constant afterwards.
  Let $N:=\dim(\SS^{\mathbf{r}_t,0}(\TT_{(0,T)}))$.

  Assume that $u_0$ is analytic in $\overline{\Omega}$ and that the right-hand side $f$ satisfies 
  \begin{align*}
    \norm{f^{(\ell)}(t)}_{L^2(\Omega)}\leq C d^{\ell} \Gamma(\ell +1)   \qquad \forall t\in [0,T], \ell \in \N_0,
  \end{align*}
  with constants $C$ and $d$ independent of $\ell$ and $t$.
  Suppose that Assumption~\ref{ass:approx_of_unitial_condtiion} is satisfied.

  Then the following error estimate holds:
  \begin{align*}
    \sqrt{\int_{0}^{T}{\norm{u_h(t) - u_{h,k}(t)}_{\widetilde{H}^s(\Omega)}^2}}
    &\lesssim  e^{-b N^{1/2}} +  e^{-b \mathcal{N}_{\Omega}^{\mu}}.
  \end{align*}
  The implied constant depends on $\Omega$, $u_0$, $f$,  $\mu$, the mesh grading and the terminal time $T$ as well as the constants from
  Assumption~\ref{ass:approx_of_unitial_condtiion}.
\end{theorem}
\begin{proof}
  The proof is  analogous to Theorem~\ref{thm:dg_fd_to_sd_1}, except we now invoke \cite[Section 5.1.2]{schoetzau_schwab}.
\end{proof}

For the model problem of smooth geometries, we can give explicit bounds for the full discretization.
\begin{corollary}
  \label{corr:hpdg_full_result}
  Assume that we are in the simplified setting of Section~\ref{sect:one_dimensional}
   and let the spaces for $\HHxy$ be designed as in Corollary~\ref{cor:error_est_cont_to_sd_energy_1d}.
  Denote the number of layers used in $\HHy$ as $M$.
  Assume that $u_0$ is analytic and $f$, $\dot{f}$ are  uniformly analytic in  a neighborhood $\widetilde{\Omega} \supset \overline{\Omega}$.

  Let $\TT_{(0,T)}:=\TT^M_{(\mathbf{0},t_1)} \cup \TT_{(t_1,T)}$
   be a mesh on $(0,T)$ which is geometrically refined towards $0$ with $M$ layers and
  has constant size for larger times $(t_1,T)$.
  We chose $\mathbf{r}_t$ such that it is linearly increasing on the geometrically refined part
  and constant afterwards.
  We take $M \sim L$, where $L$ is the number of levels used for $\HHy$.
  
  In addition, assume that the right-hand side $f$ satisfies 
  \begin{align*}
    \norm{f^{(\ell)}(t)}_{L^2(\Omega)}\leq C d^{\ell} \Gamma(\ell +1)   \qquad \forall t\in [0,T], \ell \in \N_0,
  \end{align*}
  with constants $C$ and $d$ independent of $\ell$ and $t$.

  Then there exist constants $C>0$, $b>0$ such that the following error estimate holds:
  \begin{align*}
    \sqrt{\int_{0}^{T}{\norm{u(t) - u_{h,k}(t)}_{\widetilde{H}^{s}(\Omega)}^2}}
    &\lesssim  e^{- b \left[\dim(\HHxyt)\right]^{\frac{1}{d+5}}}
  \end{align*}
  The implied constant depends on $u_0$, $f$, end time $T$, the domain $\Omega$, $\widetilde{\Omega}$, the mesh grading $\sigma$ as well as on $s$.
\end{corollary}
\begin{proof}
  Follows from Theorem~\ref{thm:dg_fd_to_sd_2}, Theorem~\ref{cor:constructive_error_est_cont_to_sd} and
  the fact that
  \begin{align*}
    \dim(\HHxyt)&\sim \dim(\HHxy) \cdot \dim\left(\SS^{\mathbf{r}_t,0}(\mathcal{T}_{(0,T)})\right) \sim  M^{d+3} \, M^2. \qedhere
    \end{align*}
\end{proof}

\subsection{Practical aspects}
\label{sect:practical_aspects}

In order to efficiently implement the scheme presented, we combine the Schur-form based approach described in~\cite{schoetzau_schwab}
  with the ideas of \cite{tensor_fem} for dealing with the extended variable.

For each time-inteval, the Schur decomposition in time leads to a sequence of problems of the form
\begin{align*}
  \sum_{j=0}^{r}{T_{ij} w_j } + \frac{k}{2} \LL_h^s w_j &= \text{r.h.s.}, \qquad {i=0,\dots r}
\end{align*}
where $T \in \C^{r\times r}$ is an upper triangular matrix. These problems can be solved using a backward-substitution,
where in each step an operator of the form $\frac{k}{\lambda_j} \LL_h^s + \id$ has to be inverted. Structurally this is
very similar to the operator $\solve^{\lambda}$, except that the parameter $\lambda:=\lambda_j/k$ is complex valued.
Proceeding like in \cite{tensor_fem} would require simultaneous diagonalization of the matrices
$$
A_{ij}:=\frac{\lambda_j}{k} v_j(0) \overline{v_i(0)} + \ltwoprodX{v_j'}{v_i'} 
\quad \text{ and } B_{ij}:=\ltwoprodX{v_j}{v_i}.
$$
Since the matrix $A$
is not  hermitean if $\Im(\lambda_j)\neq 0$, it is unclear whether this diagonalization can be done (in practice it appears to be the case). Instead
we employ the generalized Schur-form (or QZ-decomposition; see~\cite[Section 7.72]{golub_van_loan}). It gives unitary matrices
$Q$ and $Z$, such that $Q^H A Z=:T$ and $Q^H B Z=:S$ are both upper triangular. Inserting this decomposition into the definition of
$\frac{k}{\lambda_j}\LL^s_h+\id$ and using a backward-substitution leads to a sequence of problems of the form
$$\kappa_\ell \LL w_{\ell} + w_{\ell} = \text{r.h.s.}$$ for $w \in H_0^1(\Omega)$ with $\kappa_\ell \in \C$.

Overall, Problem~\ref{prob:dg_formulation} can be solved by solving
$\operatorname{dim}(\SS^{\mathbf{r}_t}(\TT_{(0,T)})) \times \operatorname{dim}(\SS^{\mathfrak{r}}(\TT^M_{(\mathbf{0},\YY)}))$ 
scalar problems posed on $\Omega$. For the case of the geometric setting of Section~\ref{sect:one_dimensional} using the
  method described in Corollary~\ref{corr:hpdg_full_result}, this means
that $\bigO(M^4)$ problems of size $\bigO(M^2)$ need to be solved.

\section{Numerical Results}
\label{sect:numerics}
In this section we test the theoretical findings of the previous sections by implementing them
using the finite element package NGSolve~\cite{ngsolve,ngsolve2} for the discretization in $\Omega$.
\subsection{Smooth solution}
\label{sect:numerics:smooth}
In order to verify our implementation, we  consider an example that has a known exact solution.
We work with the simplified model problem of Section~\ref{sect:one_dimensional}.
Namely, working in 1D, we set $\Omega:=(0,1)$, $A=I$ and $c=1$.
  The initial condition is chosen as $u_0(x):=\sin(2\pi\,x)$. As an eigenfunction of the
  Dirichlet-Laplacian this leads to the exact solution $u(x,t):=e^{-t (2\pi)^s } \sin(2\pi\,x)$. We use $s=0.5$
  and plot our findings, applying the $hp$-DG method. As seen in Figure~\ref{fig:conv_smooth}, we get the predicted exponential convergence
  with respect to the number of refinement layers.

\subsection{Singular solution}
In order to verify that our method handles startup singularities robustly, we stay in the
geometric setting of Section~\ref{sect:numerics:smooth}, but consider
the initial condition $u_0\equiv 1$ and set $s:=0.75$. We use the trivial right-hand side $f \equiv 0$.
Since the initial condition does not satisfy any compatibility condition,
we expect startup singularities. As the exact solution is unknown, we precompute a numerical solution with high accuracy using the
hp-DG method described in Corollary~\ref{corr:hpdg_full_result} with $M=13$ layers. We integrate up to the terminal time $T=1$.
Due to the predicted exponential convergence, we expect a good match of the estimated error to the (unknown) true error.

We compare different time discretization schemes. For the implicit Euler based schemes we chose a fixed polynomial degree
for discretizing $x$ and $y$ to be $p=8$. For the $hp-DG$ scheme we chose the same polynomial degree in each variable.
As an indicator for comparing the numerical cost, we use the number of systems $N$ we need  to solve involving the nonlocal
operator $\LL_h^s$. For the implicit Euler, this is proportional to the number of timesteps. For the $hp-DG$ approach
it is proportional to the number of layers $M$ squared, i.e. $N\sim M^2$.
In Figure~\ref{fig:conv_u1} we compare the spacetime $L^2$-error to the number of such systems that need solving. We see that,
as predicted, the implicit Euler method with a graded stepsize recovers the full convergence rate $\bigO(N^{-1})$ whereas
a uniform approach only yields a reduced rate. It is important to point out that practical considerations may still favor using a 
uniform grid, as in this case the corresponding matrices can be factorized only once. This yields much faster solution times in each step.
Since the reduction of order is small, the uniform approach often outperforms the graded mesh in our experience.

The best performance, as expected, is observed by the $hp-DG$ based method. It provides rapid exponential convergence of order $\bigO(e^{-b\sqrt{N}})$,
confirming Theorem~\ref{thm:dg_fd_to_sd_2} and Corollary~\ref{corr:hpdg_full_result}.

\begin{figure}
  \begin{center}
    \includeTikzOrEps{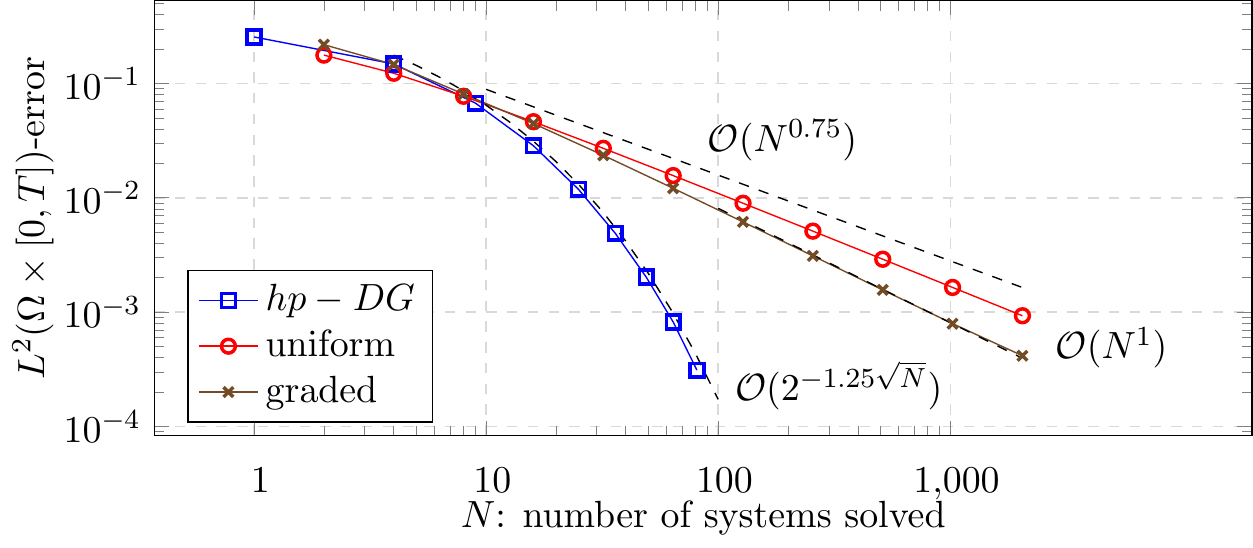}
  \end{center}
  \caption{Convergence rate in the case of non-matching initial condition}
  \label{fig:conv_u1}
\end{figure}

\subsection{A 2d example}
Although our theory in the 2D case is restricted to domains with an analytic boundary, we show numerically that the case of polygons can be successfully treated as well.
We chose $\Omega:=(0,1)^2$, $u_0\equiv 1$, $f\equiv 0$, $A:=\operatorname{I}$, $c=0$ and $s:=1/4$. Since no known analytic solution is avaliable,
we computed the approximation using $M=10$ levels of refinement in time and used it as our reference solution. All computations were done up to the
terminal time $T=1$ and using the $hp$-DG method. For the time discretization and
discretization in $y$, we used a geometric grid with $M$ layers. In $\Omega$ we used a geometrically refined grid of $3M/2 $ layers in accordance to
Corollary~\ref{thm:error_est_cont_to_sd}.

In Figure~\ref{fig:conv_2d}, we see that also in this case we get the exponential convergence with respect to the number of layers in the $hp$-refinement. This suggests that
our methods could also be extended to cover this case.

\begin{figure}
  \begin{center}
    \begin{subfigure}{0.485\textwidth}
      \includeTikzOrEps{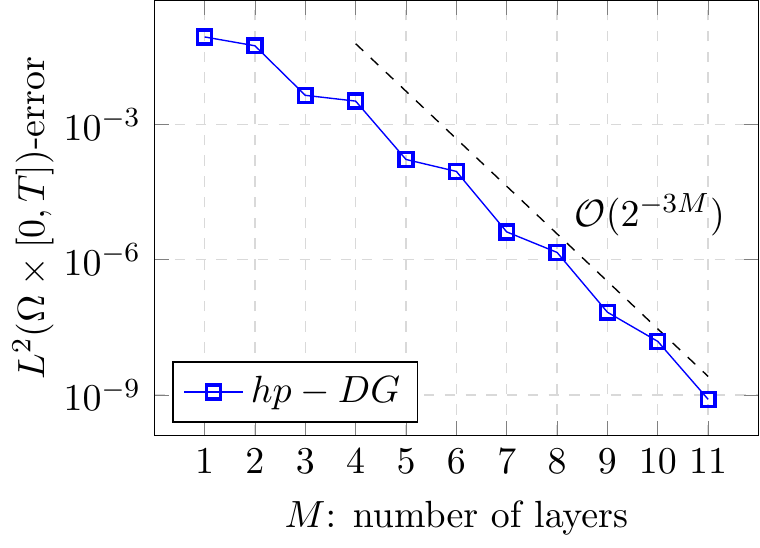}
      \caption{Convergence of the $hp$-DG method for a smooth solution}
      \label{fig:conv_smooth}
    \end{subfigure}
    \hfill
    \begin{subfigure}{0.485\textwidth}
      \includeTikzOrEps{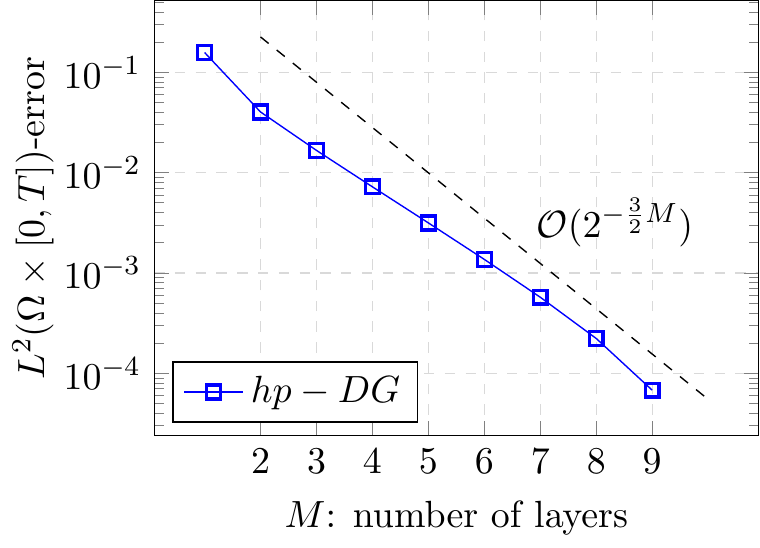}
      \caption{Convergence rate in the case of non-matching initial condition in 2D}
      \label{fig:conv_2d}
    \end{subfigure}
  \end{center}
  \caption{Convergence for the $2d$ and smooth cases} 
\end{figure}


\appendix

\section{Exponential convergence of $hp$-FEM for singularly perturbed problems with complex coefficients}
\label{appendix:oned_singular_perturbations}
In this appendix we provide the details for the regularity and approximability by high order FEM on suitably designed meshes for singularly perturbed problems with a complex perturbation parameter. We consider  on a domain 
$\Omega \subset \R^d$, $d \in \{1,2\}$ the problem of finding 
$u_\varepsilon \in H_0^1(\Omega)$ such that
\begin{align}
  \label{eq:sing_pert_model}  
    L_\varepsilon u_\varepsilon := -\varepsilon^{2} \fdiv\big(A(x,y) \nabla u_{\varepsilon}\big) + \zeta(x,y) u_{\varepsilon} &= f.
  \end{align}  
Concerning the data of this problem, we make the following assumption:
\begin{assumption}
\label{ass:sing-perturb-complex}
The domain $\Omega \subset \R^d$, $d \in \{1,2\}$ has an analytic boundary
and $\widetilde\Omega$ is a domain with $\overline{\Omega} \subset \widetilde\Omega$. 

The parameter $\varepsilon$ satisfies $\varepsilon > 0$, 
and the matrix valued function 
$A \in L^{\infty}(\Omega,\R^{d\times d})$ is analytic on 
$\widetilde{\Omega}$, pointwise and uniformly SPD. 
The function $f \in L^2(\Omega)$ is analytic on $\widetilde\Omega$. The
function $\zeta \in L^\infty(\Omega)$ is such that: 
\begin{enumerate}[(i)]
\item $\zeta(x,y) \in \widetilde{\mathcal{S}}:=\big\{ z \in \C: \abs{\pi - \operatorname{Arg}(z)} \geq \delta > 0 \big\}$,
\item $\Im(\zeta) \in \R$ is constant,
\item $\abs{\zeta(x,y)} \geq  \zeta_0 >0$ in $\Omega$,
\item $\zeta$ is analytic in $\widetilde \Omega$.
\end{enumerate}
\end{assumption}
Associated with the operator $L_\varepsilon$ is the  sesquilinear form 
$$a_{\varepsilon}(u,v):=\varepsilon^2 \int_{\Omega}{ A\nabla u \overline{\nabla v}} + \int_{\Omega} { \zeta\,u \,\overline{v}},$$
and the energy norm $\norm{\cdot}^2_{\varepsilon}:=\varepsilon^2\norm{\nabla \cdot}_{L^2(\Omega)}^2 + \norm{\cdot}_{L^2(\Omega)}^2$.

\begin{lemma}
  \label{lemma:sing_is_elliptic}
Let Assumption~\ref{ass:sing-perturb-complex} be valid. Then 
  the bilinear form $a_\varepsilon(\cdot,\cdot)$ is bounded and elliptic in the energy norm, i.e., 
  there exists $\theta(\zeta) \in (-\pi,\pi)$ such that
  \begin{align*}
    \norm{u}_{\varepsilon}^2 \lesssim \Re\left( e^{\ii \theta(\zeta)} a_{\varepsilon}(u,u) \right)  \qquad \text{ and } \qquad
    \abs{a_\varepsilon(u,v)}\lesssim \norm{u}_{ \varepsilon} \norm{v}_{\varepsilon}.
  \end{align*}
  The implied constants do not depend on $\varepsilon$ or $u$.
  This implies for the solution $u_\varepsilon$ to~\eqref{eq:sing_pert_model}:
  \begin{align}
    \label{eq:apriori_uz}
    \norm{u_\varepsilon}_{\varepsilon} \leq C \norm{f}_{L^2(\Omega)}.
  \end{align}
\end{lemma}
\begin{proof}
  For $\alpha \in \C$ with $|\alpha| = 1$, we compute:  
  \begin{align*}
     \Re\left( \alpha a_{\varepsilon}(u,u) \right)
    &= \Re( \alpha) \varepsilon^2 \norm{\nabla u}_{L^2(\Omega)}^2 + \int_{\Omega}{\Re(\alpha \zeta) \abs{u}^{2}}.
  \end{align*}
  Thus it remains to show that we can choose $\alpha$ such that $\Re(\alpha) > 0$ and $\Re(\alpha \zeta) > 0$
  uniformly in $\Omega$. If $\Im(\zeta)\geq 0$, we can pick $\alpha:=e^{-\ii \frac{\pi-\delta}{2}}$, otherwise $\alpha:=e^{\ii \frac{\pi-\delta}{2}}$
  does the trick.    
  The estimate~\eqref{eq:apriori_uz} follows from the Lax-Milgram lemma.
\end{proof}

The previous lemma ensures existence and uniqueness of solutions $u_\varepsilon$.
In the next one we further prove that $u_\varepsilon$ is analytic with explicit bounds
on the derivative with respect to the parameter $\varepsilon$.
\begin{lemma}
\label{lemma:sing_apriori}
Let Assumption~\ref{ass:sing-perturb-complex} be valid. 
Let $u_\varepsilon \in H^1_0(\Omega)$ solve \eqref{eq:sing_pert_model}
Then $u_\varepsilon$ is analytic on $\overline{\Omega}$ and satisfies:
\begin{subequations}
  \label{eq:sing_apriori}
\begin{align}
  \norm{u_\varepsilon}_{\varepsilon} &\leq C \qquad \text{and} \qquad
  \norm{\nabla^{p+2} u_{\varepsilon}}_{L^2(\Omega)}\leq C K^{p} \max(p+1,\varepsilon^{-1})^{p+2} \quad \forall p \in \N_0. 
\end{align}
\end{subequations}
\end{lemma}
\begin{proof}
  The statement is the restriction of \cite[Theorem~{2.3.1}]{melenk_book} (for 2D, in 1D the relevant result is Proposition~{2.2.1}) to the case of smooth domains.
  While this reference only considers $\zeta > 0$, the proof carries over almost verbatim. The only
  modification needed is the coercivity estimate from Lemma~\ref{lemma:sing_is_elliptic}.
\end{proof}

While Lemma~\ref{lemma:sing_apriori} will provide exponential convergence in the asymptotic case of
sufficiently large polynomial degree, the more practically relevant regime is treated using the following lemma:
\def\ubl{u_{\varepsilon}^{\textrm{BL}}}
\begin{lemma}  
  \label{lemma:oned_sing_decomposition}
Let Assumption~\ref{ass:sing-perturb-complex} be valid. 
Let $u_\varepsilon \in H^1_0(\Omega)$ solve \eqref{eq:sing_pert_model} 
for $\varepsilon \in (0,1]$. 
  Then there exists 
a smooth cut-off function $\chi$
supported by a tubular neighborhood of $\partial\Omega$ with $\chi \equiv 1$ 
in a neighborhood of $\partial\Omega$ and constants
 $C$, $\gamma$, $b > 0$ 
independent of $\varepsilon  \in (0,1] $ such that  $u_\varepsilon$ can be decomposed as
  \begin{align*}
    u_\varepsilon &= w_\varepsilon + \chi \ubl + r_{\varepsilon}
  \end{align*}
  with the following properties:
  \begin{enumerate}[(i)]
  \item \label{it:estimate_smooth_part}
    The smooth part $w_\varepsilon$ is analytic in $\Omega$ and satisfies
    $\norm{\nabla^{p} w_\varepsilon}_{L^\infty(\Omega)} \leq C \gamma^p p!$ for all $p \in \N_0$.    
  \item
    \label{it:estimate_remainder}
    The remainder $r_{\varepsilon} \in H_0^1(\Omega)$ satisfies
    $\norm{r_{\varepsilon}}_{\varepsilon} + \norm{r_{\varepsilon}}_{H^1(\Omega)} + \norm{\nabla^2 r_{\varepsilon}}_{L^2(\Omega)}\lesssim C  e^{- b/\varepsilon }$.
  \item
    \label{it:estimate_bl}
    Using boundary fitted coordinates $(\rho,\theta)$, 
    where $\rho = \operatorname{dist}(\cdot,\partial\Omega)$ and 
    $\theta$ is a parametrization of $\partial\Omega$,  the boundary layer 
    $\ubl$ can be estimated
    \begin{align*}
      \sup_{\theta} \abs{\partial_{\rho}^{n} \partial_{\theta}^{m}\ubl(\rho,\theta)}\leq C \varepsilon^{-n} \gamma^{n+m} m! e^{-\alpha \rho/\varepsilon},\;\;\rho \geq 0.
    \end{align*}
  \end{enumerate}
\end{lemma}
\begin{proof}
  We focus on the 2D case
  by adapting \cite[Theorem~{2.3.4}]{melenk_book} to the case of smooth geometries and complex data.
  The 1D case follows from adapting~\cite[Lemma~{7.1.1}]{melenk_book} instead.

  While the somewhat technical proof from~\cite{melenk_book} only considers real data $\zeta> 0$,
  it can be adapted to our setting in a mostly straight forward way. We make some comments on
  how to read the proof and how to make the required modifications.

  The construction is laid out in~\cite[Section~7]{melenk_book}. The smooth part is constructed inductively:
  \begin{align*}
    u_0:=\frac{1}{\zeta} f, \quad u_{2j+2}:=\frac{1}{\zeta} \fdiv(A \nabla u_{j}), \quad u_{2j+1}:=0 \qquad \forall j \in \N_0, \qquad
    w_{\varepsilon}:=\sum_{j=0}^{2M+1}{\varepsilon^{j} u_j}.
  \end{align*}
  The estimate (\ref{it:estimate_smooth_part}) then follows as in~\cite[Lemma 7.2.1]{melenk_book} from Cauchy's integral theorem.
  The fact that $\zeta$ is complex does not require modifications, we only need that the function $\frac{1}{\zeta}$ has an analytic extension
  to a neighborhood of $\Omega$. This is guaranteed by the assumption $\abs{\zeta} \geq \zeta_0 >0$.

  In order to construct the boundary layer function and prove~(\ref{it:estimate_bl}), one works in boundary adapted coordinates.
  Proceeding as in \cite[Section 7.3.1]{melenk_book}, $\ubl$ is defined via
  \begin{align*}
    \ubl(\theta,\rho):=\sum_{i=0}^{2M+1}{\varepsilon^{i} \widehat{U}_{i}(\theta,\widehat{\rho})}=\sum_{i=0}^{2M+1}{\varepsilon^{i} \widehat{U}_{i}(\theta,\rho/\varepsilon)},
  \end{align*}
  where $\widehat{U}_i$, using $\lambda:=\sqrt{\frac{\zeta(\theta,0)}{A_{22}(\theta,0)}}$, solves an ODE of the form
  \begin{align}    
    \label{eq:bl_ode_system}
  - \widehat{U}_i'' + \lambda^2 \widehat{U}_i = f_i, \quad \widehat{U}_{i}(0)=g, \quad \lim_{\widehat{\rho} \to \infty}{\widehat{U}_{i}} = 0.
  \end{align}
  The necessary estimates of \cite[Section~7]{melenk_book} to conclude~(\ref{it:estimate_bl}) all rely 
  on~\cite[Lemma~{7.3.6}]{melenk_book} which gives exponential decay
  for problems of the form~\eqref{eq:bl_ode_system}. It is already formulated for complex parameters $\lambda$,
  we only point out that due to the assumption that $\zeta \in \widetilde{ \mathcal{S}}$, we get $\Re(\lambda)>0$
  (using the principal branch of the complex square root, satisfying $\Re(z) \geq 0\;\; \forall z \in \C$).
  The requirement $\Re(\lambda^2) > 0$ made in 
  \cite[Lemma~{7.3.6}]{melenk_book} is not satisfied, but inspection of the proof 
  reveals that it is only needed to get unique solvability 
  of~\eqref{eq:bl_ode_system}. As seen in 
  Lemma~\ref{lemma:sing_is_elliptic}, this is also guaranteed in the 
  current setting.

  Finally, (\ref{it:estimate_remainder}) follows form the fact that  $r_{\varepsilon}:=u_{\varepsilon} - w_{\varepsilon} - \ubl$,
  solves $L_{\varepsilon} r_{\varepsilon}=f_{\varepsilon}$ where $\norm{f_\varepsilon}_{L^{\infty}(\Omega)}$ is exponentially small (see~\cite[Eqn (7.4.37)]{melenk_book}
  and $r_{\varepsilon}|_{\Gamma}=0$.
  The stated estimate then again follows from standard \textsl{a priori} estimates, most notably Lemma~\ref{lemma:sing_is_elliptic}.
\end{proof}

\begin{lemma}
  \label{lemma:approx_for_sing_pert}
  Let Assumption~\ref{ass:sing-perturb-complex} be valid.
  Let $u_\varepsilon$ solve~\eqref{eq:sing_pert_model},
  let $\TT^L_{\Omega}$ be an anisotropic geometric mesh refined towards $\partial \Omega$ as in Definition~\ref{def:anisotropic_geometric_mesh}
  and $\HHx$ be the space of continuous piecewise polynomials of degree $p$ (see~\eqref{eq:def_hhx_for_fem}).
  Assume that $\sigma^L \leq \varepsilon\leq 1$.
  
  Then there exist constants $C$, $b>0$ such that
  for all $p \in \N$
  \begin{align*}
  \inf_{v_h \in \HHx}
    \varepsilon^2 \norm{\nabla u- \nabla v_h}_{L^2(\Omega)} +\norm{u - v_h}^2_{L^2(\Omega)}&\leq C e^{- b p}.
  \end{align*}
\end{lemma}
\begin{proof}
  Analogously to~\cite[Theorem 7.7]{tensor_fem}, we note that the mesh $\TT^{L}_{\Omega}$ contains a so-called
  admissible boundary layer mesh,
  i.e. a mesh containing one layer of ``needle elements'' of size $ \lambda p \varepsilon$ (see\cite{MS98} or\cite{melenk_book} for the precise definition).

  The proof then follows completely analogously to \cite[Theorem 3.14]{MS98} (or also \cite[Theorem 3.4.8]{melenk_book}).
  The necessary ingredients to generalize to complex parameters $\zeta$ (as described in Assumption~\ref{ass:sing-perturb-complex}) are given by Lemmas~\ref{lemma:sing_apriori}
  and~\ref{lemma:oned_sing_decomposition}.
\end{proof}

\section{Polynomial liftings and interpolation spaces}
\label{sect:liftings_and_interpolation}
In this section, we investigate under which conditions we can lift discrete functions from $\HHx$ to 
functions in $\HHxy$ in a stable way. This question is deeply related to the theory of interpolation of 
discrete polynomial spaces. This can be seen in the following proposition:
\begin{proposition}[{\cite[Lemma~{40.1}]{tartar}}]
  \label{prop:general_liftings}
  Let $X_1\subseteq X_1$ be Banach spaces with continuous embedding. 
  For $\theta \in (0,1)$, denote the interpolation space by $X_\theta:=[X_0,X_1]_{\theta,2}$.
  Then the following statements hold:
  \begin{enumerate}[(i)]
    \item If $v$ is a $X_0$-valued function such that 
      $v(t) \in X_1$ and $\dot{v}(t) \in X_0$ for all $t>0$ and
      $t^{1-\theta} \norm{\dot{v}(t)}_{X_0} \in L^2(\R_+,\frac{dt}{t})$,
      $t^{1-\theta} \norm{v(t)}_{X_1} \in L^2(\R_+,\frac{dt}{t})$
      then $v(0) \in X_\theta$
      with 
      $$\norm{v(0)}^2_{X_\theta}\lesssim \int_{0}^\infty{t^{1-2\theta} \left[\norm{\dot{v}(t)}^2_{X_0} + \norm{{v}}^{2}_{X_1}\right] \,dt}.$$
    \item If $v_0 \in X_\theta$, there exists a function $v:\R_+ \to X_1$ such that
      $v(0)=v_0$ and
      \begin{align*}
        \int_{0}^\infty{t^{1-2\theta} \left[\norm{\dot{v}(t)}^2_{X_0} + \norm{{v}}^2_{X_1}\right] \,dt} \lesssim \norm{v_0}_{X_\theta}^2.       
      \end{align*}     
  \end{enumerate}
\end{proposition}
\begin{proof}
  This is just a special case of~\cite[Lemma 40.1]{tartar}. We note that in comparison to the statement in the book we
  changed the roles of $X_0$ and $X_1$. But since
  $\displaystyle  [X_1,X_0]_{\theta,2}=[X_0,X_1]_{1-\theta,2}
  $
  by~\cite[Lemma~{25.4}]{tartar}, the theorem holds in the stated form.
\end{proof}

The case of lifting a polynomial on a single element $[0,1]$ to the unit square  was addressed in~\cite{bernardi:hal-00153795}. Namely, the
following holds:
\begin{proposition}[{\cite{bernardi:hal-00153795}}]
  \label{prop:lifting_to_polynomials}
  For $d=1$, let $\QQ^p:=\operatorname{span}\{x_i, \,0\leq i \leq p\}$ denote the space of polynomials. 
  Then the following statements hold:
  \begin{enumerate}[(i)]
    \item The interpolation norm coincides with the Sobolev norm, i.e., for all $\theta \in (0,1)$
      $$\left[\left(\QQ^p,\norm{\cdot}_{L^2([0,1])}\right),\left(\QQ^p,\norm{\cdot}_{H^1([0,1])}\right)\right]_{\theta,2}
      =\left(\QQ^p,\norm{\cdot}_{{H}^{\theta}([0,1])}\right)$$
      with equivalent norms. The implied constant depends only on $\theta$.
    \item For all $u \in \QQ^p$, there exists a polynomial $\UU \in \mathcal{Q}^p([0,1]^2):=\operatorname{span}_{0\leq i,j\leq p}\big\{x^{i}  \, y^{j}\big\}$ such that
      $\trace{\UU}=u$,  $\UU(\cdot,y) \in \QQ^p $ for all $y \in [0,1]$.
      For $y>1$, $\UU$ can be extended by $0$  to $\R_+$ such that
      $\norm{\UU}_{\HHnd} \lesssim \norm{u}_{{H}^s([0,1])}$.
  \end{enumerate}
\end{proposition}
\begin{proof}
  See Corollary 4.4 and Corollary 3.3 in Chapter II of~\cite{bernardi:hal-00153795}. 
\end{proof}

When working in 2d, it is not sufficient to work only with global polyomials. It is possible to generalize
Proposition~\ref{prop:lifting_to_polynomials} to the case of piecewise polynomials on a shape regular mesh.
In 2d, this is worked out in~\cite{our_interpolation_space_paper}:
\begin{proposition}[\cite{our_interpolation_space_paper}]
  \label{prop:lifting_to_polynomials2}
  Let $\Omega \subseteq \R^2$ and $\SS^{p,1}(\TT_{\Omega})$ denote the space of piecewise polynomials
  on a shape regular grid of quadrilaterals (see Definition~\ref{def:reference-mesh}).
 
  Then the following statements hold:
  \begin{enumerate}[(i)]
    \item The interpolation norm coincides with the Sobolev norm, i.e., for all $\theta \in (0,1)$
      $$\left[\left(\SS^{p,1}(\TT_{\Omega}),\norm{\cdot}_{L^2(\Omega)}\right),\left(\SS^{p,1}(\TT_{\Omega}),\norm{\cdot}_{H^1(\Omega)}\right)\right]_{\theta,2}
      =\left(\SS^{p,1}(\TT_{\Omega}),\norm{\cdot}_{{H}^{\theta}(\Omega)}\right)$$
      with equivalent norms. The implied constant depends only on $\theta$.
    \item For all $u \in \SS^{p,1}(\TT_{\Omega})$, there exists a function
      $\UU \in C\big(\Omega \times [0,1]\big)$ such that
      $\trace{\UU}=u$,  $\UU(\cdot,y) \in \SS^{p,1}(\TT_{\Omega}) $ for all $y \in [0,1]$.
      For $y>1$, $\UU$ can be extended by $0$  to $\R_+$ such that
      $\norm{\UU}_{\HHnd} \lesssim \norm{u}_{{H}^s(\Omega)}$.
  \end{enumerate}
\end{proposition}

The previous propositions give a lifting to the space of either polynomials or continuous functions in the extended variable $y$.
Since we will be working with piecewise polynomials with a linear degree vector neither is sufficient for our needs.
We need the following variation of the previous result:
\begin{lemma}
  \label{lemma:lifting_of_polynomials}
  Let $u \in \QQ^p$ in $1d$ or $u \in \SS^{p,1}(\TT_\Omega)$,
    where $\TT_{\Omega}$ is a shape regular mesh of quadrilaterals as defined in Definition~\ref{def:reference-mesh}.
  Assume that the triangulation $\TT^L_{(\mathbf{0},\mathcal{Y})}$ satisfies $\operatorname{diam}(K_0)\leq h_{x}p^{-2}$,
  where $K_0$ is the element at $0$ and $h_x:=\min_{K \in \TT_{\Omega}}{\operatorname{diam}(K)}$.
  Then there exists a lifting $\UU_h \in  \SS^{p,1}(\TT_\Omega) \otimes \HHy$ such that
\begin{align*}
  \norm{\UU_h}_{\HHnd} \leq C \norm{u}_{{H}^{s}(\Omega)} \quad \text{ and } \quad \trace{\UU_h}=u.
\end{align*}
 The constant $C$ depends only on $s$ and the mesh grading parameter $\sigma$.
 The lifting can be chosen to be piecewise linear with respect to $y$. 
\end{lemma}
\begin{proof}
  We focus on the 2d case.
  By Propositions~\ref{prop:general_liftings} and~\ref{prop:lifting_to_polynomials2}
  , there exists a lifting
   $\UU \in C(\R_+, \SS^{p,1}(\TT_{\Omega}))$ such that 
  $$\norm{\UU}_{\HHnd} \leq C \norm{u}_{{H}^{s}(\Omega)}.$$
  Inspecting the proof of Proposition~\ref{prop:general_liftings}, as given in \cite{tartar},
  one can see that  the lifting $\UU$ is piecewise linear on the grid $\big(e^{n}\big)_{n \in \Z}$.
  By a simple rescaling, we may choose
  $\UU$ as piecewise linear in $y$ on the geometric mesh $\sigma^n$ for $n \in \Z$.
  To get a function which is in the space $\SS^{1,1}(\TT^L_{(\mathbf{0},\mathcal{Y})})$, we need to 
  make two modifications: modify $\UU$ on the element $K_0:=(0,\sigma^L)$ to also be linear and
  cut the function off at $\YY$. We define $h_0:=\operatorname{diam}(K_0)=\sigma^L$.
  
  We define $\UU_h(\cdot,t)$ as the linear interpolation between $u=\UU(0)$ and $\UU(\sigma^L)$ on $K_0$ 
  and $\UU_h=\UU$ otherwise. We need to show:
  \begin{align}    
    \int_{K_0}{y^\alpha \norm{\partial_y \UU_h(y)}^2_{L^2(\Omega)} \,dy} &\lesssim \norm{\UU}_{\HH}^2,  \label{eq:lifting_proof_l2_est} \\
    \int_{K_0}{y^\alpha \norm{\nabla_x \UU_h(y)}^2_{L^2(\Omega)} \,dy} &\lesssim \norm{\UU}_{\HH}^2. \label{eq:lifting_proof_h1_est}
  \end{align}

  We start with the first inequality. Since $\UU_h$ is the linear interpolant of $\UU$, we can write $\partial_y \UU_h= h_0^{-1} \int_{0}^{h_0}{\partial_y \UU(\tau) \,d\tau}$.
  This gives:
  \begin{multline*}    
    \int_{K_0}{y^\alpha \norm{\partial_y \UU_h(y)}^2_{L^2(\Omega)} \,dy}
    \lesssim h_0^{-2} \int_{K_0}{y^\alpha \left(\int_{0}^{h_0}{\norm{\partial_y \UU(\tau)}_{L^2(\Omega)} \,d\tau} \right)^2 \,dy} \\
    \lesssim \underbrace{h_0^{-2} \int_{K_0}{y^\alpha \left(\int_{0}^{y}{\norm{\partial_y \UU(\tau)}_{L^2(\Omega)} \,d\tau} \right)^2 \,dy}}_{=:I_1} 
      +\underbrace{ h_0^{-2} \int_{K_0}{y^\alpha \left(\int_{y}^{h_0}{\norm{\partial_y \UU(\tau)}_{L^2(\Omega)} \,d\tau} \right)^2 \,dy}}_{=:I_2}.
  \end{multline*}

  We first investigate the term $I_1$. Using the fact that $y \leq h_0$ and therefore $h_0^{-2} \leq y^{-2}$, we get
  \begin{align*}
    I_1 &\leq \int_{0}^{h_0}{y^\alpha \left(y^{-1} \int_{0}^{y}{\norm{\partial_y \UU(\tau)}_{L^2(\Omega)} \,d\tau} \right)^2 \,dy}
    \leq \int_{0}^{h_0}{y^\alpha \norm{\partial_y \UU(\tau)}^2_{L^2(\Omega)}  \,dy}
  \end{align*}
  by Hardy's inequality(see~\cite[page 28]{grisvard85a}).

  When investigating $I_2$, we distinguish $\alpha \geq 0$ and $\alpha\leq 0$.
  For $\alpha \geq 0$ we have $y^\alpha \leq \tau^\alpha$ for $y\leq \tau$ and thus after applying Cauchy Schwarz to
  get the square into the integral:
  \begin{align*}
    h_0^{-2} \int_{K_0}{y^\alpha \left(\int_{y}^{h_0}{\norm{\partial_y \UU(\tau)}_{L^2(\Omega)} \,d\tau} \right)^2 \,dy }
    &\leq  h_0^{-2} \int_{K_0}{ \left(\int_{y}^{h_0}{\tau^\alpha \norm{\partial_y \UU(\tau)}_{L^2(\Omega)} \,d\tau} \right)^2 \,dy}\\
    &\leq  \int_{y}^{h_0}{\tau^\alpha \norm{ \partial_y \UU(\tau)}^2_{L^2(\Omega)} \,d\tau}
    \leq  \norm{\UU}_{\HHnd}^2.
  \end{align*}
  For $\alpha \leq 0$, we have $h_0^{\alpha}\leq \tau^\alpha$ and get:
  \begin{multline*}
    h_0^{-2}  \int_{K_0}{y^\alpha \Big(\!\int_{y}^{h_0}{\!\norm{\partial_y \UU(\tau)}_{L^2(\Omega)} \,d\tau} \Big)^2 \!dy }
    \leq h_0^{-2} \int_{K_0}{y^\alpha \left(\int_{0}^{h_0}{h_0^{-\alpha}\tau^{\alpha} \norm{\partial_y \UU(\tau)}_{L^2(\Omega)} \,d\tau} \right)^2 \,dy } \\
    \lesssim h_0^{-2} h_0^{-\alpha} h_0^{\alpha +1}  \left(\int_{0}^{h_0}{\tau^{\alpha} \norm{\partial_y \UU(\tau)}_{L^2(\Omega)} \,d\tau} \right)^2  
    \leq  \norm{\UU}_{\HHnd}^2,
  \end{multline*}
  which proves~\eqref{eq:lifting_proof_l2_est}.

  We now show \eqref{eq:lifting_proof_h1_est}. The proof relies on an inverse estimate and the fact that $\UU_h$ approximates $\UU$.
  We estimate:
  \begin{align*}
    \int_{K_0}{y^\alpha \norm{\nabla_x \UU_h(y)}^2_{L^2(\Omega)} \,dy}
    &\lesssim \int_{K_0}{y^\alpha \norm{\nabla_x \UU_h(y)- \nabla_x \UU(y)}^2_{L^2(\Omega)} \,dy} + \int_{K_0}{y^\alpha \norm{\nabla_x \UU(y)}^2_{L^2(\Omega)} \,dy} \\
    &\lesssim \int_{K_0}{y^\alpha \norm{\nabla_x \UU_h(y)- \nabla_x \UU(y)}^2_{L^2(\Omega)} \,dy} + \norm{\UU}_{\HHnd}.
  \end{align*}
  Since $\UU_h(\cdot,y)$ and $\UU(\cdot,y)$ are (piecewise) polynomials  for all fixed $y$, we can use an inverse estimate~\cite[Theorem 3.91]{schwab_p_fem}
  to get:
  \begin{align*}
    \int_{K_0}{y^\alpha \norm{\nabla_x \UU_h(y)- \nabla_x \UU(y)}^2_{L^2(\Omega)} \,dy}
    \lesssim h_{x}^{-2}p^{4} \int_{K_0}{y^\alpha \norm{\UU_h(y)- \UU(y)}^2_{L^2(\Omega)} \,dy}.
  \end{align*}

  Since $\UU_h - \UU$ vanishes at $y=0$, we can write it as
  $
    \UU_h(y)- \UU(y)=\int_{0}^{y}{\partial_y{\UU}_h(\tau) - \partial_y {\UU}(\tau) \,d\tau}
  $
  and further estimate:
  \begin{align*}
    \int_{K_0}{y^\alpha \norm{\UU_h(y)- \UU(y)}^2_{H^1(\Omega)} \,dy}
    \lesssim h_{x}^{-2}\,p^{4}
      \int_{K_0}{y^\alpha \left( \int_{0}^{y} { \norm{\partial_y {\UU}_h(\tau)}_{L^2(\Omega)} + \norm{\partial_y \UU(\tau)}_{L^2(\Omega)} \,d\tau}\right)^2 \,dy}\\
    \lesssim 
      \underbrace{ h_{x}^{-2}\,p^{4}\int_{K_0}{y^\alpha \left( \int_{0}^{y} { \norm{\partial_y {\UU}_h(\tau)}_{L^2(\Omega)}  \,d\tau}\right)^2 \,dy} }_{=:I_3} 
    +
      \underbrace{ h_{x}^{-2}\,p^{4}\int_{K_0}{y^\alpha \left( \int_{0}^{y} { \norm{\partial_y {\UU}(\tau)}_{L^2(\Omega)}  \,d\tau}\right)^2 \,dy}}_{=:I_4}.
  \end{align*}

  The term $I_3$ is structurally analogous to the term~\eqref{eq:lifting_proof_h1_est} and can be estimated using the same techniques. The 
  extra integration in $\tau$ gives an additional power of $h_0^2$, and we get
  \begin{align*}
    I_3\leq h_{x}^{-2} p^4 h_0^2 \norm{\UU}_{\HHnd}.
  \end{align*}
  For the term $I_4$, we apply Hardy's inequality and the estimate $h_0^{-2} \leq y^{-2}$ to get:
  \begin{align*}
    I_4&=h_{x}^{-2}\,p^{4}\int_{K_0}{y^\alpha \left( \int_{0}^{y} { \norm{\partial_y {\UU}(\tau)}_{L^2(\Omega)}  \,d\tau}\right)^2 \,dy}
    \lesssim h_{x}^{-2}\, p^4 h_0^2  \int_{K_0}{y^\alpha \left( \frac{1}{y} \int_{0}^{y} { \norm{\partial_y {\UU}(\tau)}_{L^2(\Omega)}  \,d\tau}\right)^2 \,dy} \\
    &\lesssim h_{x}^{-2}\,  p^4 h_0^2  \int_{K_0}{y^\alpha  \norm{\partial_y {\UU}(\tau)}^2_{L^2(\Omega)} \,dy} =p^4 h_0^2 \norm{\UU}_{\HHnd}.
  \end{align*}  
  Overall, since we assumed $h_0\leq h_{x}\, p^{-2}$, we get the stability of the modified lifting.

  In order to get $\support{\UU_h} \subset [0,\YY]$, we pick the cutoff function $\varphi \in \SS^{1,1}(\TT^L_{(\mathbf{0},\mathcal{Y})})$ such 
  that $\varphi|_{K_i}=1$ on $K_i$ for $i=0,\dots, |\TT^L_{(\mathbf{0},\mathcal{Y})}|- 1$ and $\varphi(\YY)=0$. 
  We note that the element where $\varphi$ is non-constant has size $\bigO(1)$,  
  and it can be easily checked that $\varphi \cdot \UU_h$ is also a stable lifting of $u$.
  In order to get a function in $\SS^{1,1}\big(\TT^L_{\YY},\SS^{p,1}(\TT_{\Omega})\big)$ we interpolate the function in the grid points.
  Since $\UU_h \cdot \varphi$ is a polynomial of degree at most $2$, interpolating it down to degree 1 is stable in the $L^2$ and $H^1$ norm
  (see~\cite[Rem.~{13.5} and (13.27)]{bernardi-maday97}). Away from $0$, the weighted norms are equivalent to the standard ones.  This
  shows that the ``cutoff and interpolation''-procedure is stable in $\HH$.
\end{proof}

\section{Proof of Proposition~\ref{prop:error_up_to_ritz}}
\label{sect:proof_error_up_to_ritz}
  The following proof consists of condensed and restated results from \cite[Chapter 3]{thomee_book}.
  We fix $t_0>0$ and consider the discrete backward problem
  \begin{align}
    \label{eq_ode_for_zh}
    - \dot{z}_h + \LL^s_h z_h &= \theta,  \quad \text{ in $(0,t_0)$,} \qquad  \text{and } \qquad
    z_h(t_0) = 0.
  \end{align}
  
  For $\tau \in (0,t_0)$, we get by testing~\eqref{eq_ode_for_zh} with $\theta$ in the $L^2$-inner product
  and using~\eqref{eq:sd_ode_for_theta} and~\eqref{eq:def_ritz}:
  \begin{align*}
    \norm{\theta}_{L^2(\Omega)}^2
    &= - \ltwoprodX{\dot{z}_h(\tau)}{\theta(\tau)} + \ltwoprodX{\LL^s_h z_h(\tau)}{\theta(\tau)} \\
    &= - \frac{d}{dt} \ltwoprodX{z_h(\tau)}{u(\tau)-u_h(\tau)} +  (\rho(\tau),\dot{z}_h(\tau))_{L^2(\Omega)}. 
  \end{align*}
  For $0<\varepsilon<t_0$ we get by integrating, since $z_h(t_0)=0$:
  \begin{align}
    \int_{\varepsilon}^{t_0}{\norm{\theta(\tau)}^2_{L^2(\Omega)} \,d\tau}
    &\leq  \ltwoprodX{z_h(\varepsilon)}{u(\varepsilon)-u_h(\varepsilon)}
      + \int_{\varepsilon}^{t_0}{\norm{\rho(\tau)}_{L^2(\Omega)} \norm{\dot{z}_h(\tau)}_{L^2(\Omega)} \,d\tau} \nonumber\\
    &\leq \ltwoprodX{z_h(\varepsilon)}{u(\varepsilon)-u_h(\varepsilon)} + 
      \left(\int_{\varepsilon}^{t_0}{\norm{\rho(\tau)}^2 _{L^2(\Omega)}\,d\tau}\right)^{1/2}
      \left(\int_{\varepsilon}^{t_0}{\norm{\dot{z}_h(\tau)}^2 _{L^2(\Omega)}\,d\tau}\right)^{1/2}. \label{eq:est_for_theta_proof1}
  \end{align}
  In the limit $\varepsilon \to 0$, the first term converges due to Lemma~\ref{lemma:apriori}~(\ref{it:sg_is_l2_continuous}) to
  $$\ltwoprodX{z_h(\varepsilon)}{u(\varepsilon)-u_h(\varepsilon)} \to \ltwoprodX{z_h(0)}{u_0-u_{h,0}}.$$ 
  The following stability estimate holds for $z_h$ by Lemma~\ref{lemma:apriori}~(\ref{it:sg_apriori_stability_inh1}):
  \begin{align*}
    \int_{0}^{t_0}{\norm{\dot{z}_h(\tau)}^2_{L^2(\Omega)} \,d\tau} + t_0^{-1}\norm{z_h(0)}_{L^2(\Omega)}^2
    &\lesssim \int_{0}^{t_0}{\norm{\theta(\tau)}_{L^2(\Omega)}^2 \,d\tau}.
  \end{align*}
  Combining this estimate with \eqref{eq:est_for_theta_proof1} completes the proof of~\eqref{eq:error_up_to_ritz_l2}.

  Proof of~\eqref{eq:error_up_to_ritz_pointwise_and_hone}:
  For fixed $t>0$, testing the equation~\eqref{eq:sd_ode_for_theta} with $v:=t \, \theta(t)$ and integrating over $\Omega$ gives:
  \begin{align*}
    \frac{1}{2}\frac{d}{dt} \big( t\norm{\theta(t)}_{L^2(\Omega)}^2 \big) + t \ltwoprodX{\LL_h^s \theta(t)}{\theta(t)}
    &= t \ltwoprodX{\dot{\rho}(t)}{\theta(t)} + \frac{1}{2}\norm{\theta(t)}_{L^2(\Omega)}^2.
  \end{align*}
  We integrate in $t$ from $\varepsilon>0$ to $t$ and get:
  \begin{multline*}
    \frac{1}{2} t \norm{\theta(t)}_{L^2(\Omega)}^2 \big) + \int_{\varepsilon}^{\tau} {\ltwoprodX{\LL_h^s \theta(\tau)}{\theta(t)}\,d\tau} \\
    \leq \frac{1}{2} \varepsilon \norm{\theta(\varepsilon)}^2_{L^2(\Omega)} + 
      \sqrt{\int_{\varepsilon}^{t}{\tau^2 \norm{\dot{\rho}(\tau)}_{L^2(\Omega)}^2 \,d\tau}} \sqrt{\int_{\varepsilon}^{t}{ \norm{\theta(\tau)}_{L^2(\Omega)}^2\,d\tau}}
      + \frac{1}{2}\int_{\varepsilon}^{t}{\norm{\theta(\tau)}_{L^2(\Omega)}^2\,d\tau}.
  \end{multline*}
  We need to bound $\lim_{\varepsilon\to 0}{\varepsilon \norm{\theta(\varepsilon)}^2_{L^2(\Omega)}}$.
  Writing $\theta=\Pi_h u - u_h = \rho + u-u_h$, we use the fact that $u_h$ and $u$ are bounded by Lemma~\ref{lemma:apriori}~(\ref{it:sg_is_l2_continuous}).
  This gives:
  \begin{align*}
    \lim_{\varepsilon\to 0}{\varepsilon \norm{\theta(\varepsilon)}_{L^2(\Omega)}^2}
    &\leq \limsup_{\varepsilon\to 0}{\varepsilon \norm{\rho(\varepsilon)}_{L^2(\Omega)}^2} +  \limsup_{\varepsilon\to 0}{\varepsilon \norm{u(\varepsilon)-u_h(\varepsilon)}_{L^2(\Omega)}^2}
    &\leq  \sup_{\tau \in (0,t)}{\tau \norm{\rho(\tau)}_{L^2(\Omega)}^2}.
  \end{align*}
  By using Young's inequality and \eqref{eq:error_up_to_ritz_l2}, we easily obtain~\eqref{eq:error_up_to_ritz_pointwise_and_hone}
  from the fact that $\norm{\theta(t)}^2_{\widetilde{H}^s(\Omega)}\lesssim \ltwoprodX{\LL_h^s \theta(t)}{\theta(t)}$ by Lemma~\ref{lemma_llh_is_elliptic}.
  \clearpage

\textbf{Acknowledgments:} Financial support by the Austrian Science Fund (FWF) through
the research program ``Taming complexity in partial differential systems'' (grant SFB F65, A.R.).

\bibliographystyle{amsalpha}
\bibliography{literature}

\end{document}